\documentclass[12pt]{amsart}

\usepackage{amssymb,amsmath}
\usepackage{amsfonts}
\usepackage{amssymb}
\usepackage{amsthm}
\usepackage{caption}
\usepackage{subcaption}
\usepackage{graphicx}
\usepackage{tikz}
\usepackage[normalem]{ulem}
\usepackage{enumerate}
\usepackage{MnSymbol}
\usepackage{graphicx, scalefnt}
\usepackage[normalem]{ulem}
\usepackage{enumerate}
\usepackage[utf8]{inputenc}
\usepackage{mathtools}
\usepackage{xspace}
\usepackage{standalone}
\usepackage{calc}
\usepackage{pgfplots}
\usepackage{etex}

\usetikzlibrary{arrows,positioning,automata,shapes,calc,3d,graphs}
\usetikzlibrary{decorations.markings}
\tikzstyle{vertex}=[circle, draw, inner sep=0pt, minimum size=6pt]
\usetikzlibrary{decorations.pathmorphing}
\usetikzlibrary{patterns,decorations.pathreplacing}
\usetikzlibrary{shapes,arrows}
\usetikzlibrary{calc}

\tikzstyle{block} = [rectangle, draw, text width=5em, text centered, rounded corners, minimum height=4em]
\tikzstyle{vertex}=[circle, draw, inner sep=0pt, minimum size=10pt]
\tikzstyle{decision} = [diamond, draw, text width=5em, text badly centered, node distance=3cm, inner sep=0pt]

\theoremstyle{plain}
\newtheorem{theorem}{Theorem}[section]
\newtheorem{lemma}[theorem]{Lemma}

\newtheorem{corollary}[theorem]{Corollary}

\newtheorem{problem}[theorem]{Problem}

\newtheorem*{repp@theorem}{\repp@title (reformulated)}
\newcommand{\newrepptheorem}[2]{%
\newenvironment{repp#1}[1]{%
 \def\repp@title{#2 \ref{##1}}%
 \begin{repp@theorem}}%
 {\end{repp@theorem}}}
\makeatother
\newrepptheorem{theorem}{Theorem}

\theoremstyle{definition}

\newtheorem{example}[theorem]{Example}
\makeatletter
\newtheorem*{rep@theorem}{\rep@title \ continued}
\newcommand{\newreptheorem}[2]{%
\newenvironment{rep#1}[1]{%
 \def\rep@title{#2 \ref{##1}}%
 \begin{rep@theorem}}%
 {\end{rep@theorem}}}
\makeatother
\newreptheorem{example}{Example}

\newcommand{\cds}{{\upshape{\textbf{cds}}}\xspace} 
\newcommand{\cdr}{{\upshape{\textbf{cdr}}}\xspace} 
\definecolor{alertcolor}{HTML}{D20155}
\definecolor{dkblue}{HTML}{FF7401}
\definecolor{mdblue}{HTML}{FFBD6E}
\definecolor{ppl}{HTML}{B202D9}
\definecolor{ltgrn}{HTML}{C8FDEC}
\definecolor{ltyl}{HTML}{FCFDDC}
\definecolor{darkorange}{HTML}{A94F00}
\definecolor{experiment}{HTML}{3A1E97}
\definecolor{dkgrn}{HTML}{016444}

\title{Context Directed Reversals and the Ciliate Decryptome} 
\author{C.L. Jansen, M. Scheepers, S.L. Simon and E. Tatum}
\date{\today}


\begin{document}
\maketitle

\begin{abstract}  Prior studies of the efficiency of the block interchange (swap) and the reversal sorting operations on (signed) permutations identified specialized versions of the these operations. These specialized operations are here called context directed reversal, abbreviated \cdr, and context directed swap, abbreviated \cds. Prior works have also  characterized which (signed) permutations are sortable by \cdr or by \cds. 

It is now known that when a permutation $\pi$ is \cds sortable in $n$ steps, then any $n$ consecutive applicable \cds operations will sort $\pi$. Examples show that this is not the case for \cdr. This phenomenon is the focus of this paper. It is proven that if a signed permutation is \cdr sortable, then any \cdr fixed point of it is \cds sortable (the \cds Rescue Theorem). The \cds Rescue Theorem is discussed in the context of a mathematical model for ciliate micronuclear decryption. 

It is also known that if applications of \cds to a permutation $\pi$ reaches a \cds fixed point in $n$ steps, then any $n$  consecutive applicable \cds operations will terminate in a \cds fixed point of $\pi$. This is not the case for \cdr: It is proven that though for a given signed permutation the number of \cdr operations leading to different \cdr fixed points may be different from each other, the parity of the number of operations is the same (the \cdr Parity Theorem). This result provides a solution to two previously formulated decision problems regarding certain combinatorial games.

\end{abstract}

\section{Introduction}

The scope, efficiency and robustness of sorting algorithms are of high interest since sorting is used to prepare data for application of various information processing algorithms. The mathematical study of permutation sorting correspondingly has a long history. 

Applications of permutation sorting in biological studies have been stimulated by the discovery that often the positions of genes on the chromosomes of an organism A is a permutation of the positions of the corresponding genes on the chromosomes of another organism B. Dobzhansky and Sturtevant \cite{DS} proposed using the minimum number of reversals required to sort the gene order of organism A to that of organism B as a measure of the evolutionary distance between A and B. Hannenhalli and Pevzner \cite{HP} found an efficient algorithm for determining the minimum number of reversals when a signed permutation describes this relation between genes in A and in B. That work identified a particular class of reversals, named \emph{oriented reversals}, as instrumental to finding that minimum number. Other advances in the study of (oriented) reversals may be found in \cite{AB, BHS, TBS}. In an independent line of investigation this specific type of reversal has been postulated as one of the two sorting operations executed in ciliates during the process of converting a micronuclear precursor of a gene to its functional form. In \cite{EHPPR} and other ciliate literature this special reversal operation has been denoted \textsf{hi}, an abbreviation for \emph{hairpin inverted repeat}. We shall call this special reversal operation \emph{context directed reversal}, abbreviated \cdr.

Also block interchanges, and the special case of transpositions, have been studied as permutation sorting operations. Christie \cite{DC} studied the problem of finding the minimum number of block interchanges required to sort one permutation to another. A certain constrained class of block interchanges, called \emph{minimal block interchanges} in \cite{DC}, emerged as instrumental for efficient sorting by block interchanges. Again, independently, the model for ciliate micronuclear decryption postulates a constrained block interchange operation as a sorting operation towards accomplishing decryption of micronuclear precursors of genes. In ciliate literature, see for example \cite{EHPPR},  these constrained block interchanges are called \textsf{dlad} operations. Some instances of the minimal block interchanges identified by Christie are indeed \textsf{dlad} operations. We shall refer to a dlad operation as a \emph{context directed swap}, abbreviated \cds.

The \cds \emph{Inevitability Theorem}, a result from \cite{AHMMSTW}, shows that if a permutation is \cds sortable, then in fact indiscriminate applications of \cds will successfully sort the permutation. The advantage for the ciliate decryptome is that such a sorting operation does not require additional resources to direct strategic choices for successful sorting by \cds. However, the \cdr sorting operation does not have this inevitability feature: Indiscriminately applying \cdr sorting operations to a \cdr sortable (or reverse \cdr sortable) signed permutation may terminate in an unsuccessful sorting. This fact motivates several questions, including: (1) Do signed permutations that are \cdr sortable, but not indiscriminately \cdr sortable, actually occur in ciliates? (2) If ``{\tt yes}", how does the ciliate decryptome succeed in sorting such signed permutations? and (3) Does counting the number of \cdr operations required to sort a signed permutation $\pi$ entail finding a successful sequence of \cdr operations that sort $\pi$? 

Regarding Question (1): \cite{DP} reports the DNA sequences for the micronuclear precursors of the Actin I gene for the ciliate species \emph{Uroleptus pisces} 1 and \emph{Uroleptus pisces} 2. The first is available from \cite{genbank} under Accession Number AF508053.1, and is representable by the signed permutation
$\lbrack 1,\; 3,\;  -7,\;  -5,\;  14,\;  2,\;  4,\;  6,\;  9,\;  12,\;  -11,\;  -8,\;  13,\;  15,\;  -10\rbrack$. 
The second is available from \cite{genbank} under Accession Number AY373659.1, and is representable by the signed permutation $\lbrack  2,\;  4,\;  6,\;  9,\;  12,\;  -11,\;  -8,\;  13,\;  15,\;  -10,\; 1,\; 3,\;  -7,\;  -5,\;  14 \rbrack$.
Both are \cdr sortable, yet some sequences of applications of \cdr result in a \cdr fixed point different from the identity. This fact suggests that unless the applications of \cdr operations during micronuclear decryption follow a strategy not yet discovered in the laboratory, \cdr fixed points other than the identity would be encountered results of the decryption process.

Regarding Question (2): We prove the perhaps surprising result, Theorem \ref{cdsrescue}, named the \cds \emph{Rescue Theorem}, that any \cdr fixed point of a \cdr sortable signed permutation is a \cds sortable permutation. Thus, the \cds Inevitability Theorem of \cite{AHMMSTW} implies that the \cds sorting operation plays a rescue role in the ciliate decryptome.

Regarding Question (3): In prior work \cite{AB, HP, TBS}, finding the minimum number of reversals to sort a permutation included finding a strategic selection of \cdr operations that sort a \cdr sortable signed permutation. The \cds Inevitability Theorem plus the \cds Rescue Theorem provide another approach. The \cdr \emph{Steps Theorem}, Theorem \ref{indiscriminatecdr}, determines the number of \cdr applications required to sort a \cdr sortable signed permutation by indiscriminate applications of \cdr and \cds. 

Different \cdr fixed points of a given signed permutation are not necessarily reached in the same number of \cdr sorting steps. But the parity of the number of \cdr operations required to reach a \cdr fixed point is an invariant of each signed permutation - the \cdr \emph{Parity Theorem}, Theorem \ref{cdrparity}. This result implies a linear time solution for the \cdr Misere and the \cdr Normal Play decision problems for  combinatorial games featuring \cdr, defined previously in \cite{AHMMSTW}.

Our paper is organized as follows. After introducing basic terminology and the theoretical background information from prior work we prove the \cds Rescue Theorem and the \cdr Steps Theorem in Section 9, and the \cdr Parity Theorem in Section 10. Finally we discuss our results in the context of the ciliate micronuclear decryption model proposed in \cite{PER, PER1}.

For readers interested in the biological connections of this work we recommend the two textbooks \cite{EHPPR} and \cite{FLRTV}. 

\section{Notation and Terminology}

We consider both signed permutations and unsigned permutations. An unsigned permutation will simply be called a permutation. In either case we consider these as one-to-one and onto functions from the appropriate finite set to itself. 
For a permutation $\pi$ the notation 
\begin{equation}\label{eq:permnotation}
 \pi =   \lbrack a_1,\; \cdots,\; a_n\rbrack
\end{equation}
denotes that the permutation $\pi$ maps $a_i$ to $i$. For a signed permutation $\pi$ the notation in (\ref{eq:permnotation}) denotes that the signed permutation $\pi$ maps $a_i$ to $i$ and $-a_i$ to $-i$. The symbol $\textsf{S}_n$ denotes the set of permutations as in (\ref{eq:permnotation}), while $\textsf{S}^{\pm}_n$ denotes the set of such signed permutations.

The definitions of the \emph{context directed} sorting operations use the notion of a \emph{pointer}:
Consider a (signed) permutation 
$\alpha = \lbrack p_1, p_2, \ldots,p_i,\ldots,p_j,\ldots, p_n\rbrack$.
Consider the entry $p_i$ of $\alpha$ and say $p_i$ is the integer $k$. Then the \emph{head pointer}, or simply \emph{head} of $p_i$ is the ordered pair $(\vert k\vert,\; \vert k\vert +1)$, while the \emph{tail pointer}, or simply \emph{tail}, of $p_i$ is the ordered pair $(\vert k\vert-1,\vert k\vert)$.

In displays where pointer locations are emphasized we will use notation as follows
\[
   \alpha = \left\{\begin{tabular}{ll}
                  $\lbrack p_1, p_2, \ldots ,_{(k-1\;,k)}{p_i}_{(k,\;k+1)},\ldots ,{p_j},\ldots, p_n\rbrack$ & or alternately \\ 
                   $\lbrack p_1, p_2, \ldots ,^{t}{p_i}^{h},\ldots ,{p_j},\ldots, p_n\rbrack$ &
                  when $p_i = \vert p_i\vert  = k$, and\\
                   $\lbrack p_1, p_2, \ldots ,_{(k,\;k+1)}{p_i}_{(k-1,\;k)},\ldots ,{p_j},\ldots, p_n\rbrack$ & or alternately \\
$\lbrack p_1, p_2, \ldots ,^{h}{p_i}^{t},\ldots ,{p_j},\ldots, p_n\rbrack$ &
when $-p_i = \vert p_i\vert  = k$.\\
  \end{tabular}\right.
\]

\section{The oriented overlap graph of a signed permutation}

Consider a signed permutation $\alpha\in\textsf{S}^{\pm}_n$. 
For each pointer $(i,\; i+1)$ appearing in $\alpha$, draw an arc between the head and the tail occurrence of this pointer. One possible scenario is depicted in Figure \ref{fig:arcs}.
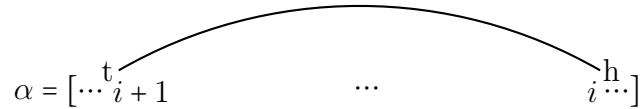
\begin{figure}[ht]
\begin{tikzpicture}[scale = 1.0]
\coordinate (v1) at (-3,3);
\coordinate(v1-) at (-3.45,3);
\coordinate (v2) at (3,3);
\coordinate (v2+) at (3.25,3);
\coordinate (v3) at (0,3);
\coordinate (v4) at (-4.1,3);
\coordinate (v5) at (3.4,3);
    \node[circle, inner sep=8pt] at (v1){$i+1$};
    \node[circle, inner sep=8pt] at (v2){$i$};
\node at (v2+) [above] {h};
\node at (v1-) [above] {t};

    \node[circle, inner sep=8pt] at (v3){$\cdots$};
    \node[circle, inner sep=8pt] at (v4){$\alpha = \lbrack \cdots$};
    \node[circle, inner sep=8pt] at (v5){$\cdots\rbrack$};
\draw [thick] (v2+)+(-0.15,.28) arc (60:120:6.40cm); 
\end{tikzpicture}
\caption{An arc between the head and tail occurences of a pointer in $\alpha$}\label{fig:arcs}
\end{figure}

The overlap graph $\mathcal{O}(\alpha) = (V,\mathcal{E})$ of signed permutation $\alpha$ has
\begin{itemize}
\item{vertex set $V$, the set of pointers of $\alpha$, and}
\item{edge set $\mathcal{E}$, the set of pairs $\{p,\; q\}$ of vertices $p$ and $q$ for which the arcs associated with $p$ and $q$ have nonempty intersection.}
\end{itemize}
Following \cite{HP, TBS} we designate a vertex of $\mathcal{O}(\alpha)$ as \emph{oriented} if the corresponding arc is between pointers of entries of opposite sign in $\alpha$. In figures oriented vertices are denoted by filled circles, while unoriented vertices are denoted by unfilled circles.

\begin{example}\label{ex:orientedcompT}
Constructing the overlap graph of the signed permutation
$T = \lbrack 1,\, -5,\, -2,\, 4,\, -3,\, 6 \rbrack$. 
\end{example}

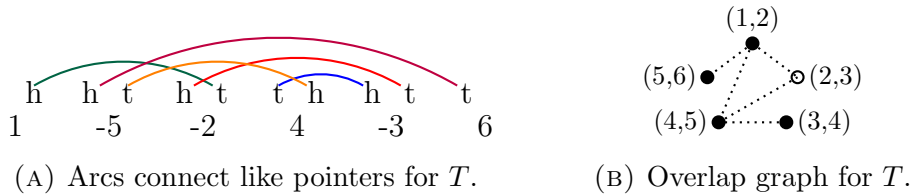
\begin{figure}[h]
\centering
\begin{subfigure}[b]{0.4\textwidth}
\begin{tikzpicture}[scale = 2.5]
\coordinate (-v4) at (-3.5,0); 
\coordinate (-v3) at (-3.0,0); 
\coordinate (-v2) at (-2.5,0); 
\coordinate (-v1) at (-2.0,0); 
\coordinate (v0) at (-1.5,0);   
\coordinate (v1) at (-1.0,0);   

\coordinate (h1) at (-3.4,-0.05);  
\coordinate (h5) at (-3.1,-0.05);  
\coordinate (t5) at (-2.9,-0.05); 
\coordinate (t2) at (-2.4,-0.05);    
\coordinate (h2) at (-2.6,-0.05);   
\coordinate (h4) at (-1.9,-0.05);    
\coordinate (t4) at (-2.1,-0.05);   
\coordinate (h3) at (-1.6,-0.05);  
\coordinate (t3) at (-1.4,-0.05); 
\coordinate (t6) at (-1.1,-0.05);    
\node at (-v4) [below] {1};
\node at (-v3) [below] {-5};
\node at (-v2) [below] {-2};
\node at (-v1) [below] {4};
\node at (v0) [below] {-3};
\node at (v1) [below] {6};

\node at (h1) [above] {h};
\node at (h2) [above] {h};
\node at (t2) [above] {t};
\node at (h3) [above] {h};
\node at (t3) [above] {t}; 
\node at (h4) [above] {h}; 
\node at (t4) [above] {t};
\node at (t5) [above] {t};
\node at (h5) [above] {h};
\node at (t6) [above] {t};

\draw [thick,  color=dkgrn] (t2)+(-0.05,.15) arc (60:120:0.95cm); 
\draw [thick,  color=red] (t3)+(-0.05,.15) arc (60:120:1.10cm); 
\draw [thick,  color=blue] (h3)+(-0.05,.15) arc (60:120:0.45cm); 
\draw [thick,  color=orange] (h4)+(-0.05,.15) arc (60:120:0.95cm); 
\draw [thick,  color=purple] (t6)+(-0.05,.15) arc (60:120:1.90cm); 
\end{tikzpicture}
\caption{Arcs connect like pointers for $T$. }\label{fig:OverlapGraphT1}
\end{subfigure}
\begin{subfigure}[b]{0.4\textwidth}
\centering
\begin{tikzpicture}[scale = .15]
\coordinate (p12) at (0,6);
\coordinate (p23) at (4,3); 
\coordinate (p34) at (3,-1); 
\coordinate (p45) at (-3,-1); 
\coordinate (p56) at (-4,3); 

\node at (p12) [above] {\footnotesize (1,2)};
\node at (p23) [right] {\footnotesize (2,3)};
\node at (p34) [right] {\footnotesize (3,4)};
\node at (p45) [left] {\footnotesize (4,5)};
\node at (p56) [left] {\footnotesize (5,6)};

\draw [thick, fill = black] (p12) circle (15 pt);
\draw [thick] (p23) circle (15 pt);
\draw [thick, fill = black] (p34) circle (15 pt);
\draw [thick, fill = black] (p45) circle (15 pt);
\draw [thick, fill = black] (p56) circle (15 pt);

\draw [thick, dotted] (p12) -- (p23);%
\draw [thick, dotted] (p23) -- (p45); %
\draw [thick, dotted] (p34) -- (p45);   %
\draw [thick, dotted] (p12) -- (p45);   %
\draw [thick, dotted] (p12) -- (p56);    %

\end{tikzpicture}
\caption{Overlap graph for $T$.}\label{fig:OverlapGraphFCS0}
\end{subfigure}
\caption{Construction of the overlap graph for $T$.}\label{fig:OverlapGraphFCS1}
\end{figure}

Let $\mathcal{G} = (V,\mathcal{E})$ be a graph. Declare vertices $x$ and $y$ of $\mathcal{G}$ \emph{in reach} if there are vertices $x_0,\; x_1,\; \cdots,\; x_m$ such that $x=x_0$ and $y= x_m$, and for each $i<m$ we have $\{x_i,\; x_{i+1}\}\in\mathcal{E}$. The ``in reach" relation is an equivalence relation on the vertex set of $\mathcal{G}$. If an equivalence class of this relation has more than one element it is called a \emph{component} of $\mathcal{G}$. If this equivalence relation has only one equivalence class, we say that $\mathcal{G}$ is a \emph{connected} graph. 
If $(\mathcal{G},\; f)$ is an oriented graph, then a component of $\mathcal{G}$ is said to be an \emph{oriented component} if some vertex belonging to the component is oriented. The member of an equivalence class with only one element is said to be an \emph{isolated} vertex.

\begin{example}\label{ex:unorientedcomp}
The signed permutation
$S = \lbrack-6 ,\;3 ,\;-4 ,\;2 ,\;5 ,\;-1 ,\;7 ,\;9 ,\;8,\;10 \rbrack$. 
The overlap graph of $S$ follows:
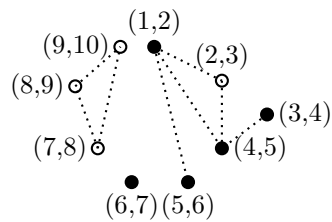
\begin{figure}[ht]
\begin{tikzpicture}[scale = .15]
\coordinate (p12) at (-4,6);
\coordinate (p23) at (2,3); 
\coordinate (p34) at (6,0); 
\coordinate (p45) at (2,-3); 
\coordinate (p56) at (-1,-6); 
\coordinate (p67) at (-6,-6); 
\coordinate (p78) at (-9,-3);
\coordinate (p89) at (-11,2.5); 
\coordinate (p910) at (-7,6);

\node at (p12) [above] {\footnotesize (1,2)};
\node at (p23) [above] {\footnotesize (2,3)};
\node at (p34) [right] {\footnotesize (3,4)};
\node at (p45) [right] {\footnotesize (4,5)};
\node at (p56) [below] {\footnotesize (5,6)};
\node at (p67) [below] {\footnotesize (6,7)};
\node at (p78) [left] {\footnotesize (7,8)};
\node at (p89) [left] {\footnotesize (8,9)};
\node at (p910) [left] {\footnotesize (9,10)};

\draw [thick, fill = black] (p12) circle (15 pt);
\draw [thick] (p23) circle (15 pt);
\draw [thick, fill = black] (p34) circle (15 pt);
\draw [thick, fill = black] (p45) circle (15 pt);
\draw [thick, fill = black] (p56) circle (15 pt);
\draw [thick, fill = black] (p67) circle (15 pt);
\draw [thick] (p78) circle (15 pt);
\draw [thick] (p89) circle (15 pt);
\draw [thick] (p910) circle (15 pt);

\draw [thick, dotted] (p12) -- (p23);%
\draw [thick, dotted] (p23) -- (p45); %
\draw [thick, dotted] (p34) -- (p45);   %
\draw [thick, dotted] (p12) -- (p45);   %
\draw [thick, dotted] (p12) -- (p56);    %
\draw [thick, dotted] (p78) -- (p89); %
\draw [thick, dotted] (p78) -- (p910); %
\draw [thick, dotted] (p89) -- (p910);%

\end{tikzpicture}
\caption{Overlap graph for $S$.}\label{fig:OverlapGraphFCS2}
\end{figure}\\
The oriented graph in Figure~\ref{fig:OverlapGraphFCS2}  has an isolated vertex (that is oriented), an unoriented component, and an oriented component.
\end{example}

\section{Context directed reversals}

Let a signed permutation $\alpha$ with entries $x$ and $y$ as well as a pointer $(i,i+1)$  be given. Suppose that the entries $x$ and $y$ of $\alpha$ each is associated with the pointer $(i,i+1)$\footnote{Thus, $(i,\;i+1)$ is the head pointer of one, and the tail pointer of the other, and $\{i,\; i+1\} = \{\vert x\vert,\; \vert y\vert\}$.}. We define $\beta = {\text{\cdr}}_{(i,i+1)}(\alpha)$ by cases:

{\flushleft{\underline{Case 1:}}} $\alpha = \lbrack \cdots,\; _{(i,\;i+1)}x,\; \cdots,\; y_{(i,\;i+1)},\; \cdots\rbrack$ or $\alpha = \lbrack \cdots,\; x_{(i,\;i+1)},\; \cdots,\; _{(i,\;i+1)}y,\; \cdots\rbrack$\\
In this case $x\cdot y > 0$\footnote{In other words, $x$ and $y$ have the same sign.}, and we define $\beta = \alpha$. 

{\flushleft{\underline{Case 2:}}} $\alpha = \lbrack \cdots,\; a,\;{\mathbf{_{(i,\;i+1)}x,\dots,\;{b}}},\; _{(i,\;i+1)}y,\; \cdots\rbrack$.\\
In this case $\beta$ is obtained from $\alpha$ by reversing the order of all entries starting at $x$ and ending at $b$, and changing the signs of these entries. Thus,
\[
   \beta = {\text{\cdr}}_{(i,i+1)}(\alpha) = \lbrack \cdots,\; a,\; {\mathbf{-b,\dots,\; -x_{(i,\;i+1)}}},\; _{(i,\;i+1)}y,\; \cdots\rbrack
\]
{\flushleft{\underline{Case 3:}}} $\alpha = \lbrack \cdots,\; x_{(i,\;i+1)},\; {\mathbf{a,\;\dots,\; y_{(i,\;i+1)}}},\;b,\; \cdots\rbrack$.\\
In this case $\beta$ is obtained from $\alpha$ by reversing the order of all entries starting at $a$ and ending at $y$, and changing the signs of these entries. Thus,
\[
   \beta = {\text{\cdr}}_{(i,i+1)}(\alpha) = \lbrack \cdots,\; x_{(i,\;i+1)},\;{\mathbf{ _{(i,\;i+1)}-y,\;\dots,\; -a}},\; b,\cdots,\; \rbrack
\]

\begin{example} In the signed permutation $\pi = \lbrack -2,\; 1,\; -4,\; 3\rbrack$ the pointer $(2,\, 3)$ is the head pointer of $-2$ and the tail pointer of $3$, entries with opposite sign. Thus the context directed reversal $ {\text{\cdr}}_{(2,3)}$ is applicable to $\pi$, and ${\text{\cdr}}_{(2,3)}(\pi) = \lbrack 4,\; -1,\; 2,\; 3\rbrack$.
\end{example}

Context directed reversals have been called \emph{oriented reversals} in \cite{HP,TBS}, and \textbf{hi}, the hairpin inversion in \cite{EHPPR} and related literature.
A fundamental theorem from \cite{HP} implies
\begin{theorem}[Hannenhalli-Pevzner]\label{cdrsortfundthm}
If the overlap graph of the signed permutation $\alpha$ has no unoriented components, then $\alpha$ is sortable by applications of \text{\cdr}. 
\end{theorem}

Another fundamental theorem from \cite{HP} implies  (see the remarks around \cite{TBS} Theorem 2):
\begin{theorem}[Hannenhalli-Pevzner]\label{cdrfundthm}
If the signed permutation $\alpha$ is sortable by applications of \text{\cdr}, then all sequences of successive \text{\cdr} operations that sort $\alpha$ are of the same length.
\end{theorem}

\section{Local complementation of oriented graphs, and the operation \textbf{gcdr}}

The operation \cdr on a signed permutation can be simulated by a corresponding operation on the overlap graph of the signed permutation. This operation on arbitrary graphs, introduced in the prior work \cite{AB, HP, TBS}, is described next.

For a graph $\mathcal{G} = (V,\; \mathcal{E})$ and a vertex two-coloring $f:V\rightarrow\{0,\; 1\}$ the pair $(\mathcal{G},\; f)$ is said to be an \emph{oriented} graph. Vertices $v$ with $f(v)=0$ are said to be \emph{oriented vertices}, while the vertices $v$ with $f(v)=1$ are called \emph{unoriented} vertices.
For a given set $S$ and an oriented graph $(\mathcal{G},f)$ we define the oriented graph
\[
  (\mathcal{G}^{\prime},\; f^{\prime}) = \textsf{lc}((\mathcal{G},\; f),\; S)
\]
as follows: 
\begin{enumerate}
\item{$\mathcal{G}^{\prime} = (V^{\prime},\mathcal{E}^{\prime})$, where $V^{\prime} = V$ and $\mathcal{E}^{\prime} = (\mathcal{E}\setminus \lbrack S\rbrack^{2}) \cup (\lbrack V\cap S\rbrack^{2}\setminus \mathcal{E})$.}
\item{$f^{\prime}:V^{\prime}\rightarrow\{0,\; 1\}$ is defined by
   \[
      f^{\prime}(v) = \left\{\begin{tabular}{ll}
                                          $1 - f(v)$ & if $v \in S$ \\
                                          $f(v)$       & otherwise 
                                        \end{tabular}
                               \right.
   \]
}
\end{enumerate}
The oriented graph $(\mathcal{G}^{\prime},f^{\prime})$ is said to be the $S$-\emph{local complement} of oriented graph $(\mathcal{G},f)$. It is evident that when $S\cap V = \emptyset$, then $(\mathcal{G},f) = \textsf{lc}((\mathcal{G},\; f),\; S)$. Thus we may assume that $S\subseteq V$.
An example of an oriented graph on seven vertices is given in Figure \ref{fig:SampleOrientedGraph}. 
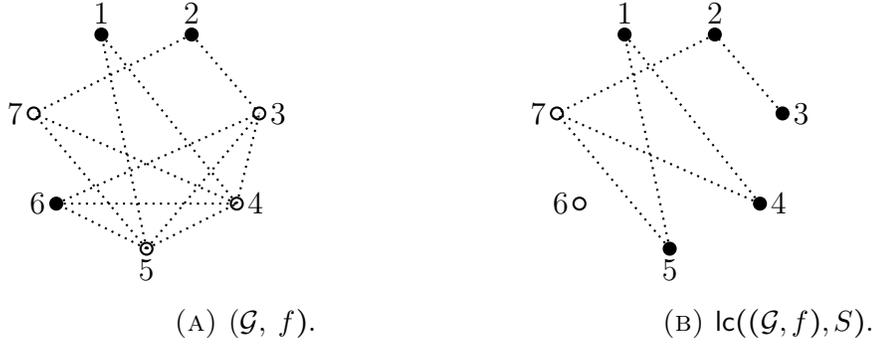
\begin{figure}[h!]
\centering
\begin{tabular}{cc}
\begin{subfigure}[b]{0.4\textwidth}
\begin{tikzpicture}[scale = .15]
\coordinate (green) at (-4,9);
\coordinate (red) at (4,9); 
\coordinate (orange) at (10,2); 
\coordinate (blue) at (8,-6); 
\coordinate (purple) at (0,-10); 
\coordinate (gray) at (-8,-6); 
\coordinate (yellow) at (-10,2);

\node at (green) [above] {1};
\node at (red) [above] {2};
\node at (orange) [right] {3};
\node at (blue) [right] {4};
\node at (purple) [below] {5};
\node at (gray) [left] {6};
\node at (yellow) [left] {7};

\draw [thick, fill = black] (green) circle (15 pt);
\draw [thick, fill = black] (red) circle (15 pt);
\draw [thick] (orange) circle (15 pt);
\draw [thick] (blue) circle (15 pt);
\draw [thick] (purple) circle (15 pt);
\draw [thick, fill = black] (gray) circle (15 pt);
\draw [thick] (yellow) circle (15 pt);

\draw [thick,dotted] (purple) -- (green);%
\draw [thick,dotted] (blue) -- (green);   %
\draw [thick,dotted] (red) -- (orange);   %
\draw [thick,dotted] (red) -- (yellow);    %
\draw [thick,dotted] (blue) -- (orange); %
\draw [thick,dotted] (purple) -- (orange);%
\draw [thick,dotted] (gray) -- (orange);  %
\draw [thick,dotted] (blue) -- (yellow);  %
\draw [thick,dotted] (purple) -- (blue);  %
\draw [thick,dotted] (gray) -- (blue);     %
\draw [thick,dotted] (purple) -- (yellow);%
\draw [thick,dotted] (purple) -- (gray);%

\end{tikzpicture}
\caption{($\mathcal{G},\; f)$. }\label{fig:SampleOrientedGraph}
\end{subfigure}
&
\begin{subfigure}[b]{0.4\textwidth}
\begin{tikzpicture}[scale = .15]
\coordinate (red) at (0,-10);       
\coordinate (orange) at (-8,-6); 
\coordinate (yellow) at (-10,2);  
\coordinate (green) at (-4,9);     
\coordinate (blue) at (4, 9);       
\coordinate (purple) at (10,2);  
\coordinate (gray) at (8,-6);     

\node at (red) [below] {5};
\node at (orange) [left] {6};
\node at (yellow) [left] {7};
\node at (green) [above] {1};
\node at (blue) [above] {2};
\node at (purple) [right] {3};
\node at (gray) [right] {4};

\draw [thick, fill = black] (red) circle (15 pt);
\draw [thick, fill = white] (orange) circle (15 pt);
\draw [thick, fill = white] (yellow) circle (15 pt);
\draw [thick, fill = black] (green) circle (15 pt);
\draw [thick, fill = black] (blue) circle (15 pt);
\draw [thick, fill = black] (purple) circle (15 pt);
\draw [thick, fill = black] (gray) circle (15 pt);

\draw [thick,dotted] (red) -- (green);
\draw [thick,dotted] (red) -- (yellow);
\draw [thick,dotted] (blue) -- (yellow);
\draw [thick,dotted] (gray) -- (green);
\draw [thick,dotted] (purple) -- (blue);
\draw [thick,dotted] (gray) -- (yellow);

\end{tikzpicture}
\caption{$\textsf{lc}((\mathcal{G},f),S)$.}\label{fig:SampleLocCompl}

\end{subfigure}
\end{tabular}
\caption{Oriented vertices marked in black. $S$ is the set $\{3,\; 4,\; 5,\; 6\}$.}\label{fig:SampleLocComplMain}
\end{figure}
 The resulting graph $\textsf{lc}((\mathcal{G},f),S)$ is displayed in Figure \ref{fig:SampleLocCompl}.

Let $(\mathcal{G},f)$ be an oriented graph. Consider a vertex $p$ of $\mathcal{G}$. If $p$ is not an oriented vertex of $\mathcal{G}$, then declare $(\mathcal{G}^{\prime},\;f^{\prime}) = (\mathcal{G},\; f)$. However, if $p$ is an oriented vertex of $\mathcal{G}$, then define: 
\[ 
  \textsf{N}(p) = \{q:\; \{p,\; q\} \mbox{ is an edge of } \mathcal{G}\}\bigcup\{p\}. 
\]
Finally, define
\[
  \textbf{gcdr}((\mathcal{G},f),p) = \textbf{lc}((\mathcal{G},f), \textsf{N}(p)).
\]
In \cite{TBS} the authors use the notation $\mathcal{G}/p$ for $\textbf{gcdr}((\mathcal{G},f),p)$.
Observe that if in $\mathcal{G}$ the vertex $p$ is oriented, then in $\textbf{gcdr}((\mathcal{G},f),p)$ the vertex $p$ is unoriented and isolated.

\section{Simulating \textbf{cdr} on the oriented overlap graph of a signed permutation}

We next relate an application of the \textbf{cdr} sorting operation on a signed permutation $\alpha$ to an application of \textbf{gcdr} on the oriented overlap graph of $\alpha$.

\begin{example}\label{ex:onecomp}
The signed permutation $\alpha = \lbrack 3,-8,-2,5,1,-7,4,6\rbrack$.
\end{example}

\begin{figure}[h!]
\begin{subfigure}[b]{0.4\textwidth}
\centering
\begin{tikzpicture}[scale = .10]
\coordinate (green) at (-4,9);
\coordinate (red) at (4,9); 
\coordinate (orange) at (10,2); 
\coordinate (blue) at (8,-6); 
\coordinate (purple) at (0,-10); 
\coordinate (gray) at (-8,-6); 
\coordinate (yellow) at (-10,2);

\node at (green) [above] {(1,2)};
\node at (red) [right] {(2,3)};
\node at (orange) [right] {(3,4)};
\node at (blue) [right] {(4,5)};
\node at (purple) [below] {(5,6)};
\node at (gray) [left] {(6,7)};
\node at (yellow) [left] {(7,8)};

\draw [thick, fill = black] (green) circle (20 pt);
\draw [thick, fill = black] (red) circle (20 pt);
\draw [thick] (orange) circle (20 pt);
\draw [thick] (blue) circle (20 pt);
\draw [thick] (purple) circle (20 pt);
\draw [thick, fill = black] (gray) circle (20 pt);
\draw [thick] (yellow) circle (20 pt);

\draw [thick,dotted] (purple) -- (green);%
\draw [thick,dotted] (blue) -- (green);   %
\draw [thick,dotted] (red) -- (orange);   %
\draw [thick,dotted] (red) -- (yellow);    %
\draw [thick,dotted] (blue) -- (orange); %
\draw [thick,dotted] (purple) -- (orange);%
\draw [thick,dotted] (gray) -- (orange);  %
\draw [thick,dotted] (blue) -- (yellow);  %
\draw [thick,dotted] (purple) -- (blue);  %
\draw [thick,dotted] (gray) -- (blue);     %
\draw [thick,dotted] (purple) -- (yellow);%
\draw [thick,dotted] (purple) -- (gray);%

\end{tikzpicture}
\caption{The overlap graph $(\mathcal{G},f)$ of  $\alpha$.}\label{fig:Samplecdrmovegraph}
\end{subfigure}
\begin{subfigure}[b]{0.4\textwidth}
\centering
\begin{tikzpicture}[scale = .10]
\coordinate (green) at (-4,9);
\coordinate (red) at (4,9); 
\coordinate (orange) at (10,2); 
\coordinate (blue) at (8,-6); 
\coordinate (purple) at (0,-10); 
\coordinate (gray) at (-8,-6); 
\coordinate (yellow) at (-10,2);

\node at (green) [above] {(1,2)};
\node at (red) [right] {(2,3)};
\node at (orange) [right] {(3,4)};
\node at (blue) [right] {(4,5)};
\node at (purple) [below] {(5,6)};
\node at (gray) [left] {(6,7)};
\node at (yellow) [left] {(7,8)};

\draw [thick, fill = black] (green) circle (20 pt);
\draw [thick, fill = black] (red) circle (20 pt);
\draw [thick, fill = black] (orange) circle (20 pt);
\draw [thick, fill = black] (blue) circle (20 pt);
\draw [thick, fill = black] (purple) circle (20 pt);
\draw [thick, fill = white] (gray) circle (20 pt);
\draw [thick, fill = white] (yellow) circle (20 pt);

\draw [thick,dotted] (purple) -- (green);%
\draw [thick,dotted] (blue) -- (green);   %

\draw [thick,dotted] (red) -- (orange);   %
\draw [thick,dotted] (red) -- (yellow);    %

\draw [thick,dotted] (blue) -- (yellow);  %
\draw [thick,dotted] (purple) -- (yellow);%

\end{tikzpicture}
\caption{Overlap graph of $\text{\cdr}_{(6,7)}(\alpha)$.}\label{fig:Samplemoveongraph}
\end{subfigure}
\caption{$\text{\cdr}_{(6,7)}(\alpha)$ in Figure \ref{fig:Samplemoveongraph} agrees with $\textbf{gcdr}((G,f),(6,7))$.}\label{fig:cdrandgcdr}
\end{figure}

Consider $\mathcal{M}_n$, the set of oriented graphs on the set of $n$ pointers associated with elements of $\textsf{S}^{\pm}_n$. Let 
\[
   \Xi_n:\textsf{S}^{\pm}_n \longrightarrow \mathcal{M}_n
\]
denote the map that associates with each signed permutation $\alpha\in\textsf{S}^{\pm}_n$ its corresponding oriented overlap graph. 
The following theorem has been proven in \cite{HP}, and it shows that  the diagram in Figure \ref{cdrcommutativediagram} commutes:
\begin{theorem}\label{cdrgcdr} Let $n$ be a positive integer and let $\alpha$ be a signed permutation in $\textsf{S}^{\pm}_n$ and let $p$ be a pointer of $\alpha$. Then
$\Xi_n(\textbf{cdr}_p(\alpha)) = \textbf{gcdr}(\Xi_n(\alpha),p)$.
\end{theorem}
A good exposition of this result is given in Section 2 of \cite{TBS}.

\tikzstyle{line} = [draw, -latex']

\begin{figure}[ht]
\centering
  \begin{tikzpicture}[node distance = 3cm, auto]
    \node (init) {$\alpha$}; 
    \node [below of=init] (mginit) {$\Xi_n(\alpha)$};

    \node [right of=init, xshift = 2cm] (init2) {$\alpha^{\prime}$};
    \node [below of=init2] (mginit2) {$\Xi_n(\alpha^{\prime})$};

    \path [line] (init) -- node[align = center] {$\Xi_n$}(mginit);
    \path [line] (init2) -- node[align = center] {$\Xi_n$} (mginit2);
    \path [line] (init) -- node[align = center] {$\textbf{cdr}_{p}$} (init2);
    \path [line] (mginit) -- node[align = center] {$\textbf{gcdr}(\cdot,\; p)$} (mginit2);
  \end{tikzpicture} 
\caption{}\label{cdrcommutativediagram}
\end{figure}

\section{Oriented Components of graphs and \textbf{gcdr}.}

The following result, Theorem 1 of \cite{TBS}, generalizes Theorem 4 of \cite{HP} from the context of oriented graphs arising from signed permutations, to the general context of oriented graphs. 
\begin{theorem} [Fundamental Theorem of Oriented Graphs]\label{gcdrtheorem}
If each component of the oriented graph $(\mathcal{G},f)$ is oriented,  then there exists an oriented vertex $v$ of $\mathcal{G}$ such that each component of $\textbf{gcdr}((\mathcal{G},f),\; v)$ is oriented.
\end{theorem} 

Let $(\mathcal{G},f)$ be a finite oriented graph, and let $(v_1,\; \cdots,\; v_k)$ be a sequence of vertices of $\mathcal{G}$. Define $(\mathcal{G}_1,\;f_1) = \textbf{gcdr}((\mathcal{G},f),\;v_1)$ and for $i<k$, $(\mathcal{G}_{i+1},\; f_{i+1}) = \textbf{gcdr}((\mathcal{G}_i,\; f_i),\; v_{i+1})$. Following \cite{TBS} we say that $(v_1,\; \cdots,\; v_k)$ is a \emph{sequence of oriented vertices} if $v_1$ is an oriented vertex of $(\mathcal{G},f)$ and for each $i<k$, $v_{i+1}$ is an oriented vertex of $(\mathcal{G}_i,f_i)$. An oriented sequence of vertices is a \emph{maximal sequence} if $(\mathcal{G}_{k},f_{k})$ has no oriented vertices. It is a \emph{total sequence} if $(\mathcal{G}_{k},f_{k})$ has only isolated, unoriented, vertices.

Observe that a total sequence is a maximal sequence, but a maximal sequence need not be a total sequence. Theorem \ref{gcdrtheorem} implies 
\begin{corollary}\label{gcdrtotal}
If each component of the oriented finite graph $(\mathcal{G},f)$ is oriented, then there is a total sequence of oriented vertices in $\mathcal{G}$.
\end{corollary}

\begin{problem}\label{totalsequenceproblem}
Let $(\mathcal{G},f)$ be a connected oriented graph with an oriented vertex. Are all total sequences of oriented vertices for $\mathcal{G}$ of the same length?
\end{problem}

\section{The \cdr Revision Theorem.}  


We now reformulate Theorem 3 from \cite{TBS}, named here the \emph{\cdr Revision Theorem}, for our purposes. Throughout this section we assume that $(\mathcal{G}, f)$ is an oriented graph. We also allow for the presence of a number of isolated vertices. The vertex set of $\mathcal{G}$ is $V$, while the edge set is $E$. 


\begin{theorem} [\cdr Revision Theorem, \cite{TBS}]\label{tbstheorem}
Let $(\mathcal{G},f)$ be a graph which has a unique component. Assume that this component is oriented. If $(v_1,\; v_2,\, \cdots,\; v_{m-1},\;v_m)$ is a maximal, but not total, sequence of oriented vertices for the graph $(\mathcal{G},f)$, then there exists an $\ell<m$ and a sequence $(r_1,\;r_2)$ of vertices of $\mathcal{G}$ such that:
\begin{enumerate}
\item{$(v_1,\;\cdots,\;v_{\ell},\;r_1,\;r_2,\; v_{\ell+1},\;\cdots,\; v_m)$ is a maximal oriented sequence in $(\mathcal{G},f)$, and }
\item{for each $j$ such that $\ell< j\le m$ neither $r_1$, nor $r_2$ is an oriented vertex in any of the graphs $\mathbf{gcdr}(\mathbf{gcdr}(\cdots\mathbf{gcdr}((\mathcal{G},f),v_1),\cdots v_{j-1}) ,v_{j})$.
}
\end{enumerate}
\end{theorem}

Repeated applications of the \cdr Revision Theorem then leads to 
\begin{corollary}\label{tbscorollary}
Let $(\mathcal{G},f)$ be a graph that has a unique component. If this component is oriented and $(v_1,\; v_2,\, \cdots,\; v_{m-1},\;v_m)$ is a maximal, but not total, sequence of oriented vertices for $(\mathcal{G},f)$, then there exists an $\ell<m$ and a sequence $(r_1,\;r_2,\; \cdots,\; r_{2k-1},\;r_{2k})$ of vertices of $\mathcal{G}$ such that $(v_1,\;\cdots,\;v_{\ell},\;r_1,\;r_2,\;\cdots,\;r_{2k-1},\;r_{2k},\; v_{\ell+1},\;\cdots,\; v_m)$ is a total oriented sequence in $(\mathcal{G},f)$.
\end{corollary}

\begin{example}\label{ex:maxseq}
Consider the signed permutation
$\pi = \lbrack 1,\, 3,\, 5,\, -2,\, -6,\, 4 \rbrack$. 
\begin{figure}[h]
\begin{subfigure}[b]{0.4\textwidth}
\centering
\begin{tikzpicture}[scale = 2.5]
\coordinate (-v4) at (-3.5,0); 
\coordinate (-v3) at (-3.0,0); 
\coordinate (-v2) at (-2.5,0); 
\coordinate (-v1) at (-2.0,0); 
\coordinate (v0) at (-1.5,0);   
\coordinate (v1) at (-1.0,0);   

\coordinate (h1) at (-3.4,-0.05);  
\coordinate (t3) at (-3.1,-0.05);   
\coordinate (h3) at (-2.9,-0.05);  
\coordinate (t5) at (-2.6,-0.05);   
\coordinate (h5) at (-2.4,-0.05);  
\coordinate (t2) at (-1.9,-0.05);    
\coordinate (h2) at (-2.1,-0.05);   
\coordinate (t6) at (-1.4,-0.05);  
\coordinate (t4) at (-1.1,-0.05);    
\coordinate (h4) at (-0.9,-0.05);   

\node at (-v4) [below] {1};
\node at (-v3) [below] {3};
\node at (-v2) [below] {5};
\node at (-v1) [below] {-2};
\node at (v0) [below] {-6};
\node at (v1) [below] {4};

\node at (h1) [above] {h};
\node at (h2) [above] {h};
\node at (t2) [above] {t};
\node at (h3) [above] {h};
\node at (t3) [above] {t}; 
\node at (h4) [above] {h}; 
\node at (t4) [above] {t};
\node at (t5) [above] {t};
\node at (h5) [above] {h};
\node at (t6) [above] {t};

\draw [thick,  color=dkgrn] (h2)+(-0.05,.15) arc (60:120:0.95cm); 
\draw [thick,  color=blue] (t2)+(-0.05,.15) arc (60:120:1.45cm); 
\draw [thick,  color=red] (h4)+(-0.05,.15) arc (60:120: 1.65cm); 
\draw [thick,  color=orange] (t4)+(-0.05,.15) arc (60:120:1.75cm); 
\draw [thick,  color=purple] (t6)+(-0.05,.15) arc (60:120:0.95cm); 
\end{tikzpicture}
\caption{Arcs connect like pointers of $\pi$. }\label{fig:OverlapGraphpi1}
\end{subfigure}
\begin{subfigure}[b]{0.4\textwidth}
\centering
\begin{tikzpicture}[scale = .10]
\coordinate (green) at (-4,9); 
\coordinate (red) at (4,9); 
\coordinate (orange) at (10,2); 
\coordinate (blue) at (0,-6); 
\coordinate (purple) at (-10,2); 

\node at (green) [above] {(1,2)};
\node at (red) [right] {(2,3)};
\node at (orange) [right] {(3,4)};
\node at (blue) [right] {(4,5)};
\node at (purple) [below] {(5,6)};

\draw [thick, fill = black] (green) circle (20 pt);
\draw [thick, fill = black] (red) circle (20 pt);
\draw [thick] (orange) circle (20 pt);
\draw [thick] (blue) circle (20 pt);
\draw [thick, fill = black] (purple) circle (20 pt);

\draw [thick,dotted] (green) -- (orange);%
\draw [thick,dotted] (green) -- (blue);    %
\draw [thick,dotted] (green) -- (purple);%
\draw [thick,dotted] (red) -- (orange);   %
\draw [thick,dotted] (red) -- (blue);    %
\draw [thick,dotted] (red) -- (purple);    %
\draw [thick,dotted] (orange) -- (blue); %

\end{tikzpicture}
\caption{The overlap graph of  $\pi$.} \label{fig:maxseqgraph}
\end{subfigure}
\caption{Constructing the overlap graph of $\pi$}\label{fig:OverlapGraphpi2}
\end{figure}
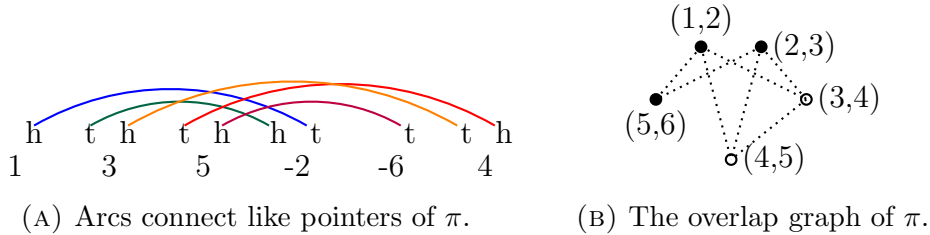
The oriented sequence $((5,\, 6))$ is maximal. Each two-term sequence in the first column of Figure \ref{fig:orientedtomaximal} is an oriented sequence in the overlap graph of $\pi$, while the corresponding three term sequence in the second column is maximal.  \begin{figure}[ht]
\begin{center}
\begin{tabular}{|l|cl|} \hline
$((1,2),\, (2,3))$ & & $((1,2),\, (2,3),\, (5,6))$\\ 
$((1,2),\, (3,4))$ & & $((1,2),\, (3,4),\, (5,6))$\\ 
$((1,2),\, (4,5))$ & & $((1,2),\, (4,5),\, (5,6))$\\ 
$((2,3),\, (1,2))$ & & $((2,3),\, (1,2),\, (5,6))$\\ 
$((2,3),\, (3,4))$ & & $((2,3),\, (3,4),\, (5,6))$\\ 
$((2,3),\, (4,5))$ & & $((2,3),\, (4,5),\, (5,6))$\\ \hline
\end{tabular}
\end{center}
\caption{Oriented, and corresponding maximal sequences for $\pi$.}\label{fig:orientedtomaximal}
\end{figure}
By Theorem \ref{tbstheorem}, for each of these maximal three-term sequences there is a pair of vertices of the overlap graph that can be inserted right before the term $(5,6)$ to produce a five-term maximal oriented sequence. One example of such a five-term maximal sequence, $((1,2),\, (3,4),\, (4,5),\, (2,3),\, (5,6))$, is a total sequence for the overlap graph of $\pi$. By Theorem~\ref{cdrfundthm} all total sequences for the overlap graph of $\pi$ are of length 5.
\end{example}

\section{The \cds rescue theorem.}

As can be gleaned from Example \ref{ex:maxseq}, even when a signed permutation is \cdr-sortable, not all applications of \cdr-sorting operations successfully sort the permutation.
A failed sorting during micronuclear decryption in a ciliate could have devastating consequences for the organism. In this section we prove that the hypothesized presence of also the \cds sorting operation is highly advantageous for accomplishing successful sorting of signed permutations in the organism.

For the reader's convenience we review information regarding \cds. 
Suppose $\alpha$ has a pair of pointers $p = (x,\; x+1)$, $q = (y,\; y+1)$. Then $\beta= {\textbf{cds}}_{p,\; q}(\alpha)$,  the result of the sorting operation is as follows:
{\flushleft{\underline{Case 1:}}} If the pointers do not appear in the order $p \ldots q \ldots p \ldots q$, or  in the order $q \ldots p \ldots q \ldots p$ we define
$\beta = {\text{\cds}}_{p,\; q}(\alpha) = \alpha$.
{\flushleft{\underline{Case 2:}}} If
$\alpha = \lbrack \cdots,\; a,\; {\mathbf{_px+1,\; b,\;\dots,c,\;\; y_q}},\;d,\;\cdots,\; e,\; x_p,\; {\mathbf{f,\; \dots,\; g}},\;_qy+1,\;h,\; \cdots \rbrack$, define
\[
  \beta = {\text{\cds}}_{p,\; q}(\alpha) = \lbrack \cdots,\; a,\;{\mathbf{f,\; \dots,\; g}},\;d,\;\cdots,\; e,\; x_p,\; {\mathbf{_px+1,\; b,\;\dots,c,\;\; y_q}} ,\;_qy+1,\;h,\;\cdots \rbrack
\]
{\flushleft{\underline{Case 3:}}} If
$\alpha = \lbrack \cdots,\; a,\;x_p,\; {\mathbf{ b,\;\dots,c,\;\; y_q}},\;d,\;\cdots,\;e,\;  {\mathbf{ _px+1,\; f,\; \dots,\; g}},\;_qy+1,\;h,\;\cdots \rbrack$, define
\[
  \beta = {\text{\cds}}_{p,\; q}(\alpha) = \lbrack \cdots,\; a,\; x_p,\; {\mathbf{_px+1,\; f,\; \dots,\; g}},\;d,\;\cdots,\;e,\;  {\mathbf{ b,\;\dots,c,\;\; y_q}} ,\;_qy+1,\;h,\;\cdots \rbrack
\]
{\flushleft{\underline{Case 4:}}} If
$\alpha = \lbrack \cdots,\; a,\;x_p,\; {\mathbf{ b,\;\dots,c}},\; _qy+1,\;d,\;\cdots,\; e,\; {\mathbf{ _px+1,\; f,\; \dots,\; g,\;y_q}},\;h,\;\cdots \rbrack$, define
\[
  \beta = {\text{\cds}}_{p,\; q}(\alpha) = \lbrack \cdots,\; a,\; x_p,\; {\mathbf{_px+1,\; f,\; \dots,\; g,\; y_q}},\;_qy+1,\;d,\;\cdots,\; e,\; {\mathbf{ b,\;\dots,c}},\;h,\;\cdots \rbrack
\]
{\flushleft{\underline{Case 5:}}} If
$\alpha = \lbrack \cdots,\; a,\; {\mathbf{_px+1,\; b,\;\dots,c}},\; _qy+1,\;d,\;\cdots,\;e,\;x_p,\; {\mathbf{ f,\; \dots,\; g,\;y_q}},\;h,\;\cdots \rbrack$, define
\[
  \beta = {\text{\cds}}_{p,\; q}(\alpha) = \lbrack \cdots,\; a,\; {\mathbf{ f,\; \dots,\; g,\; y_q}},\;_qy+1,\;d,\;\cdots,\;e,\;x_p,\; {\mathbf{ _px+1,\; b,\;\dots,c}},\;h,\;\cdots \rbrack
\]
\vspace{0.05in}

\begin{example}\label{ex:cds} We perform ${\text{\cds}}_{(3,4),(6,7)}$ on $\alpha = [3,6,5,2,4,8,1,7]$: 
Label $\alpha$ with the relevant pointers:
$\alpha = [3_{(3,4)},{\mathbf{6}_{(6,7)}},5,2,{}_{(3,4)}{\mathbf{4,8,1}},{}_{(6,7)}7]$.
 Performing the swap yields $\beta = [3,{\mathbf{4,8,1}},5,2,{\mathbf{6}},7]$. 
Both $3,4$ and $6,7$ are now adjacent in $\beta$.
\end{example}

\begin{lemma}[The \textsf{cdr}-\textsf{cds} Lemma]\label{cdsrescuelemma}
Let a, b, and c be pointers in a signed permutation $\pi$.  Assume that $\beta = \text{\cdr}_c\circ\text{\cdr}_b\circ\text{\cdr}_a(\pi)$. Also assume that neither $\text{\cdr}_a$, nor $\text{\cdr}_b$ applies to $\gamma = \text{\cdr}_c(\pi)$. Then $\text{\cds}_{(a,b)}$ applies to $\gamma$ and $\beta = \text{\cds}_{(a,b)}(\gamma)$.
\end{lemma}

\begin{proof}  
By hypothesis neither $\text{\cdr}_a$ nor $\text{\cdr}_b$ applies to $\gamma = \text{\cdr}_c(\pi)$. Thus in $\gamma$ all occurrences of the pointer $a$ are of the same sign, and similarly for the pointer $b$. As $\text{\cdr}_a$ applies to the signed permutation $\pi$, the two occurrences of pointer $a$ in $\pi$ have opposite signs. Thus in $\pi$ one of these occurences, and not both, of $a$ is flanked by  pointer $c$, as depicted in Figure \ref{fig:cstraddlesa}.

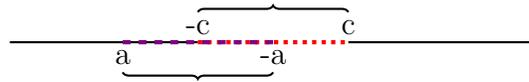
\begin{figure}[h]
  \centering
\begin{tikzpicture}[scale=.5]
\coordinate (c) at (-4,0);
\coordinate (k) at (1,0);
\node at (k)  [yshift=0.2cm]{-c};
\coordinate (k+) at (5,0);
\node at (k+)  [yshift=0.2cm]{c}; 
\coordinate (xc) at (10,0);

\coordinate (i) at (-1,0); 
\node at (i)  [yshift=-0.2cm]{a}; 
\coordinate (i+) at (3,0); 
\node at (i+)  [yshift=-0.2cm]{-a};

only<2->{\node[ inner sep=2pt] at (k){};}
\node[inner sep=2pt] at (k+){};

only<2->{\draw [black, thick] (c) -- (k);}
only<2->{\draw [red, ultra thick, dotted] (k) -- (k+);}
only<2->{\draw [black, thick] (k+) -- (xc);}

only<2->{\draw [violet, ultra thick, dashed] (i) -- (i+);}

\draw [
    thick,
    decoration={
        brace, 
        raise=0.40cm
    },
    decorate
] (k.east) -- (k+) ;

\draw [
    thick,
    decoration={
        brace,
        mirror,
        raise=0.40cm
    },
    decorate
] (i.east) -- (i+) ;
\end{tikzpicture}  
\caption{Pointers $a$ and $c$ in signed permutation $\pi$.}\label{fig:cstraddlesa}
\end{figure}

Figure \ref{fig:cstraddlesa} depicts one of eight configurations regarding placement of and signs on $a$ and $c$. We shall give the argument for this depiction, leaving the similar arguments for the other cases to the reader. Now consider the possible placements of the pointer $b$. There are three cases, depicted in Figure \ref{fig:cdsrescuecases}. We explain the first case, leaving the other two to the reader.

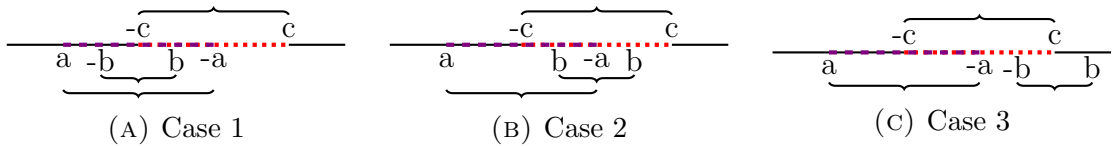
\begin{figure}[h]
  \begin{subfigure}{.3\textwidth} 
  \centering
\begin{tikzpicture}[scale=.5]
\coordinate (c) at (-2.5,0);
\coordinate (k) at (1,0);
\node at (k)  [yshift=0.2cm]{-c}; 
\coordinate (k+) at (5,0);
\node at (k+)  [yshift=0.2cm]{c}; 
\coordinate (xc) at (6.5,0);

only<2->{\node[ inner sep=2pt] at (k){};}
\node[inner sep=2pt] at (k+){};

\coordinate (i) at (-1,0); 
\node at (i)  [yshift=-0.2cm]{a}; 
\coordinate (i+) at (3,0); 
\node at (i+)  [yshift=-0.2cm]{-a}; 

\coordinate (j) at (0,0); 
\node at (j)  [yshift=-0.2cm]{-b}; 
\coordinate (j+) at (2,0); 
\node at (j+)  [yshift=-0.2cm]{b}; 
only<2->{\draw [black, thick] (c) -- (k);}
only<2->{\draw [red, ultra thick, dotted] (k) -- (k+);}
only<2->{\draw [black, thick] (k+) -- (xc);}

only<2->{\draw [violet, ultra thick, dashed] (i) -- (i+);}

\draw [
    thick,
    decoration={
        brace, 
        raise=0.40cm
    },
    decorate
] (k.east) -- (k+) ;

\draw [
    thick,
    decoration={
        brace,
        mirror,
        raise=0.60cm
    },
    decorate
] (i.east) -- (i+) ;

\draw [
    thick,
    decoration={
        brace,
        mirror,
        raise=0.40cm
    },
    decorate
] (j.east) -- (j+) ;
\end{tikzpicture}  
\caption{Case 1}
  \end{subfigure}
  \begin{subfigure}{.3\textwidth} 
  \centering
\begin{tikzpicture}[scale=.5]
\coordinate (c) at (-2.5,0);
\coordinate (k) at (1,0);
\node at (k)  [yshift=0.2cm]{-c}; 
\coordinate (k+) at (5,0);
\node at (k+)  [yshift=0.2cm]{c}; 
\coordinate (xc) at (6.5,0);

only<2->{\node[ inner sep=2pt] at (k){};}
\node[inner sep=2pt] at (k+){};

\coordinate (i) at (-1,0); 
\node at (i)  [yshift=-0.2cm]{a}; 
\coordinate (i+) at (3,0); 
\node at (i+)  [yshift=-0.2cm]{-a}; 

\coordinate (j) at (2,0); 
\node at (j)  [yshift=-0.2cm]{b}; 
\coordinate (j+) at (4,0); 
\node at (j+)  [yshift=-0.2cm]{b}; 
only<2->{\draw [black, thick] (c) -- (k);}
only<2->{\draw [red, ultra thick, dotted] (k) -- (k+);}
only<2->{\draw [black, thick] (k+) -- (xc);}

only<2->{\draw [violet, ultra thick, dashed] (i) -- (i+);}

\draw [
    thick,
    decoration={
        brace, 
        raise=0.40cm
    },
    decorate
] (k.east) -- (k+) ;

\draw [
    thick,
    decoration={
        brace,
        mirror,
        raise=0.60cm
    },
    decorate
] (i.east) -- (i+) ;

\draw [
    thick,
    decoration={
        brace,
        mirror,
        raise=0.40cm
    },
    decorate
] (j.east) -- (j+) ;
\end{tikzpicture}  
\caption{Case 2}
  \end{subfigure}
  \begin{subfigure}{.3\textwidth} 
  \centering
\begin{tikzpicture}[scale=.5]
\coordinate (c) at (-2.5,0);
\coordinate (k) at (1,0);
\node at (k)  [yshift=0.2cm]{-c}; 
\coordinate (k+) at (5,0);
\node at (k+)  [yshift=0.2cm]{c}; 
\coordinate (xc) at (6.5,0);

only<2->{\node[ inner sep=2pt] at (k){};}
\node[inner sep=2pt] at (k+){};

\coordinate (i) at (-1,0); 
\node at (i)  [yshift=-0.2cm]{a}; 
\coordinate (i+) at (3,0); 
\node at (i+)  [yshift=-0.2cm]{-a}; 

\coordinate (j) at (4,0); 
\node at (j)  [yshift=-0.2cm]{-b}; 
\coordinate (j+) at (6,0); 
\node at (j+)  [yshift=-0.2cm]{b}; 
only<2->{\draw [black, thick] (c) -- (k);}
only<2->{\draw [red, ultra thick, dotted] (k) -- (k+);}
only<2->{\draw [black, thick] (k+) -- (xc);}

only<2->{\draw [violet, ultra thick, dashed] (i) -- (i+);}

\draw [
    thick,
    decoration={
        brace, 
        raise=0.40cm
    },
    decorate
] (k.east) -- (k+) ;

\draw [
    thick,
    decoration={
        brace,
        mirror,
        raise=0.40cm
    },
    decorate
] (i.east) -- (i+) ;

\draw [
    thick,
    decoration={
        brace,
        mirror,
        raise=0.40cm
    },
    decorate
] (j.east) -- (j+) ;
\end{tikzpicture}  
\caption{Case 3}
  \end{subfigure}

\caption{Three general cases for placement of pointers $a$, $b$ and $c$ in $\pi$.}\label{fig:cdsrescuecases}
\end{figure}

{\flushleft{\bf Case 1 of Figure \ref{fig:cdsrescuecases}:} }
We write $\pi$ as the concatenation
\[
  \pi = \pi_1\,\pi_2\, \pi_3\, \pi_4\, \pi_5\, \pi_6\, \pi_7
\]
where: $\pi_1$ is the segment up to and including pointer $a$, $\pi_2$ is the segment from pointer $a$ up to and including pointer $-b$, $\pi_3$ is the segment from pointer $-b$ up to and including pointer $-c$, $\pi_4$ is the segment from pointer $-c$ up to and including pointer $b$, $\pi_5$ is the segment from pointer $b$ up to and including pointer $-a$, $\pi_6$ is the segment from pointer $-a$ up to and including pointer $c$, and $\pi_7$ is the final segment.

Applying $\text{\cdr}_a$ produces
$\pi_1\, -\pi_5\, -\pi_4\, -\pi_3\, -\pi_2\, \pi_6\, \pi_7$.
Then applying $\text{\cdr}_b$ produces $\pi_1\, -\pi_5\, \pi_3\, \pi_4\, -\pi_2\, \pi_6\, \pi_7$.
Finally applying $\text{\cdr}_c$ produces
\[
  \beta = \pi_1\, -\pi_5\, \pi_3\, -\pi_6\, \pi_2\, -\pi_4\, \pi_7.
\]
On the other hand, first applying $\text{\cdr}_c$ to $\pi$ produces
\[
  \gamma = \pi_1\, \pi_2\, \pi_3\, -\pi_6\, -\pi_5\, -\pi_4\, \pi_7.
\]
By tracking the movement of the pointers $a$ and $b$ and their signs in this step, we see that the segments $\pi_2$ and $\pi_5$ are exchanged when $\text{\cds}_{a,b}$ is applied to $\gamma$, producing $\beta$.

Cases 2 and 3 of Figure \ref{fig:cdsrescuecases} are treated similarly: Partition $\pi$ into the seven corresponding segments, and track the movement of the pointers as the sorting operations are applied.
\end{proof}

\begin{repexample}{ex:maxseq} 
Application of $\text{\cdr}_{(5,6)}$ to $\pi=\lbrack 1,\, 3,\, 5,\, -2,\, -6,\, 4 \rbrack$ produces $\gamma = \lbrack 1,\, 3,\, 5,\, 6,\, 2,\, 4 \rbrack$, to which no further context directed reversals can be applied. This corresponds to $((5,6))$ being a maximal oriented sequence in the overlap graph of $\pi$. The three term oriented sequence $((1,2),\, (2,3),\, (5,6))$ is also maximal, and applying \cdr for these pointers in the order listed in the sequence produces $\delta = \lbrack 1,\, 2,\, 3,\, 5,\,6,\, 4\rbrack$. Observe that $\text{\cds}_{(1,2),(2,3)}$ is applicable to $\gamma$, and that $\delta = \text{\cds}_{(1,2),(2,3)}(\gamma)$.\\
\end{repexample}

Lemma \ref{cdsrescuelemma} combined with the \cdr Revision Theorem, Theorem \ref{tbstheorem}, and its corollary, leads to the \cds Rescue Theorem. In preparation for the proof of this theorem we first give the following two lemmas:

\begin{lemma}\label{max} Let $w\in V$ be an oriented vertex of $(\mathcal{G},f)$. 
The following are equivalent:
\begin{enumerate}
\item{$(w)$ is a maximal oriented sequence.}
\item{A vertex $v$ of $(\mathcal{G},f)$ is oriented if, and only if, $\{w,\;v\}$ is an edge of $(\mathcal{G},f)$.}
\end{enumerate}
\end{lemma}
\begin{proof}
{\flushleft{$(1)\Rightarrow(2)$}:} Assume that $(w)$ is a maximal oriented sequence. Consider any $v\not\in N(w)$. Then the orientation of $v$ in the graph $(\mathcal{G}_1,f_1) = \textbf{gcdr}((\mathcal{G},f),w)$ is unchanged from its orientation in $(\mathcal{G},f)$. As $(w)$ is a maximal sequence of oriented vertices, it follows that $v$ is unoriented in $(\mathcal{G}_1,f_1)$, and thus in $(\mathcal{G},f)$. For the rest of $(2)$, we assume that vertex $v$ is not oriented in $(\mathcal{G},f)$. If $v$ were an element of $N(w)$, then in $(\mathcal{G}_1,f_1)$ $v$ would be an oriented vertex. But this contradicts the maximality of $(w)$. 
{\flushleft{$(2)\Rightarrow(1)$}:} Assume that, contrary to $(2)$, there is an oriented vertex $v$ of $\mathcal({G},f)$ such that $\{v,\; w\}$ is not an edge of $\mathcal{G}$. Then $v$ is still an oriented vertex in $(\mathcal{G}_1,f_1)$, whence the oriented sequence $(w)$ is not maximal.
\end{proof}

It follows that if $(w)$ is a maximal sequence in an oriented graph, then no sequence of the form $(a,w)$ of vertices of the oriented graph can be an oriented sequence. Moreover, if an oriented graph has a maximal sequence $(w)$, then only one of its components is oriented.

\begin{lemma}\label{maxextend} Let $w$ be a vertex in the oriented graph $(\mathcal{G},f)$. If $(w)$ is a maximal sequence, then any oriented sequence of the form $(u,v,w)$ of $(\mathcal{G},f)$ is maximal.
\end{lemma}
\begin{proof}
Assume that there is an oriented sequence $(u,v,w)$ that is not maximal. Choose a vertex $x$ such that $(u,v,w,x)$ is an oriented sequence. Put $(\mathcal{G}_1,f_1) = \textbf{gcdr}((\mathcal{G},f),u)$, $(\mathcal{G}_2,f_2) = \textbf{gcdr}((\mathcal{G}_1,f_1),v)$, $(\mathcal{G}_3,f_3) = \textbf{gcdr}((\mathcal{G}_2,f_2),w)$. 
As $u$ is an oriented vertex of $(\mathcal{G},f)$ (by hypothesis) Lemma \ref{max} implies:
\begin{equation}\label{winu}
w \in N(u) \mbox{ in } (\mathcal{G},f).
\end{equation}
But then $w$ is unoriented in $(\mathcal{G}_1,f_1)$, while it is oriented in $(\mathcal{G}_2,f_2)$, implying 
\begin{equation}\label{winu}
w \in N(v) \mbox{ in } (\mathcal{G}_1,f_1).
\end{equation}
Suppose that $v\not\in N(u)$ in $(\mathcal{G},f)$. Then as the orientation of $v$ in $(\mathcal{G}_1,f_1)$ is the same as its orientation in $(\mathcal{G},f)$, we find that $v$ is an oriented vertex of $(\mathcal{G},f)$ and thus in $N(w)$, by Lemma \ref{max}.  
If, on the other hand, $v\in N(u)$ in $(\mathcal{G},f)$, then as the orientation of $v$ in $(\mathcal{G}_1,f_1)$ is opposite to its orientation in $(\mathcal{G},f)$, we find that $v$ is an unoriented vertex of $(\mathcal{G},f)$ and thus not a member of $N(w)$, by Lemma \ref{max}. Thus, we must consider the following two cases, (A) and (B):

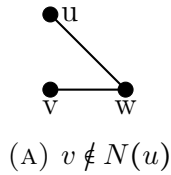
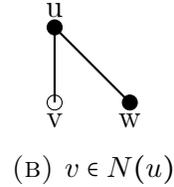
\begin{figure}[h]
\begin{subfigure}{.5\textwidth}
  \centering
  \begin{tikzpicture}[scale=.5]
\coordinate (a) at (0,0);
\coordinate (b) at (2,0);
\coordinate (c) at (0,2);
\node at (a) [below]{v};
\node at (b) [below]{w};
\node at (c)[right]{u};

only<2->{\node[draw,circle, inner sep=2pt, fill] at (a){};}
\node[draw,circle, inner sep=2pt, fill] at (b){};
\node[draw,circle, inner sep=2pt, fill] at (c){};

only<2->{\draw [black, thick] (a) -- (b);}
only<2->{\draw [black, thick] (b) -- (c);}
\end{tikzpicture} \caption{$v\not\in N(u)$}
\end{subfigure}%
\begin{subfigure}{.5\textwidth}
  \centering
\begin{tikzpicture}[scale=.5]
\coordinate (a) at (0,2);
\coordinate (b) at (0,0);
\coordinate (c) at (2,0);
\node at (a) [above]{u};
\node at (b) [below]{v};
\node at (c)[below]{w};

{\node[draw,circle, inner sep=2pt, fill] at (a){};}
\node[draw,circle, inner sep=2pt] at (b){};
\node[draw,circle, inner sep=2pt,fill] at (c){};

\draw [black, thick] (a) -- (b);
\draw [black, thick] (a) -- (c);    
\end{tikzpicture}\caption{$v \in N(u)$} \label{Case 1}
\end{subfigure}
\end{figure}

{\flushleft{\underline{Case (A), $v\not\in N(u)$: }}} There are two subcases regarding $x$:

{\flushleft{\underline{Case 1: $x\in N(w)$: }}} 
By Lemma \ref{max} $x$ is an oriented vertex in $(\mathcal{G},f)$. We consider the four possibilities $x\not\in N(u)\cup N(v)$, $x\in N(u)\setminus N(v)$, $x\in N(v)\setminus N(u)$ and $x\in N(v)\cap N(u)$. 

\begin{enumerate}
\item{\underline{$x\not\in N(u)\cup N(v)$:} In this case $x$ is oriented in $(\mathcal{G}_2,f_2)$ and still has an edge with $w$, so that $x$ is unoriented in $(\mathcal{G}_3,f_3)$.}
\item{\underline{$x\in N(u)\setminus N(v)$:} Now $u$, $w$ and $x$ are oriented and pairwise have an edge, so that in $(\mathcal{G}_1,f_1)$ these three vertices are unoriented and $x$ has no edge with any of $u$, $v$ or $w$. But then $x$ is unoriented also in $(\mathcal{G}_3,f_3)$.}
\item{\underline{$x\in N(v)\setminus N(u)$:} In this case in $(\mathcal{G}_1,f_1)$ the vertices $x$, $v$ and $w$ pairwise have edges, $w$ is unoriented and $v$ and $x$ are oriented. But then in $(\mathcal{G}_2,f_2)$ there is no edge among any two of the vertices $x$, $w$ and $v$, and $x$ is unoriented. It follows that $x$ is unoriented also in $(\mathcal{G}_3,f_3)$.}
\item{\underline{$x\in N(u)\cap N(v)$:} In this case the sets of vertices $\{u,\; x,\; w\}$ and $\{v,\; x,\; w\}$ each pairwise has edges in $(\mathcal{G},f)$, and all are oriented vertices. Then in $(\mathcal{G}_1,f_1)$ the edges $\{x,\; v\}$ and $\{w,\; v\}$ are still present, but $x$ and $w$ are both unoriented while $v$ is oriented. Thus in $(\mathcal{G}_2,f_2)$ $w$ and $x$ are oriented and there is an edge $\{w,\; x\}$. But then $x$ is unoriented in $(\mathcal{G}_3,f_3)$.}
\end{enumerate}

{\flushleft{\underline{Case 2: $x\not\in N(w)$}}} Similarly, a consideration of cases shows that also in this case we would have a contradiction with the assumption that $(u,\; v,\; w,\; x)$ is an oriented sequence.
\begin{enumerate}
\item{\underline{$x\not\in N(u)\cup N(v)$:} In this case $x$ is unoriented in $(\mathcal{G}_3,f_3)$.}
\item{\underline{$x\in N(u)\setminus N(v)$:} Then $x$ is oriented in $(\mathcal{G}_1,\; f_1)$ and a member of $N(w)$ of the now unoriented $w$. As we are in case (A), also $v\in N(w)$ and $v$ is oriented. But then in $(\mathcal{G}_2,f_2)$ both $w$ and $x$ are oriented and $\{w,\; x\}$ is an edge in $(\mathcal{G}_2,f_2)$. It follows that $x$ is unoriented in $(\mathcal{G}_3,f_3)$.}
\item{\underline{$x\in N(v)\setminus N(u)$:} Now in $(\mathcal{G}_1,f_1)$ $x$ is still a member of $N(v)$ and unoriented, so that in $(\mathcal{G}_2,f_2)$ both $x$ and $w$ are oriented and $\{w,\; x\}$ is an edge. But then $x$ is unoriented in $(\mathcal{G}_3,f_3)$.}
\item{\underline{$x\in N(u)\cap N(v)$:} In $(\mathcal{G}_1,f_1)$ $x$ and $v$ are oriented while $w$ is unoriented, and there is an edge between any two of $\{w,\; x,\; v\}$. But then in $(\mathcal{G}_2,f_2)$ $x$ is unoriented, and there is no edge between $w$ and $x$. It follows that in $(\mathcal{G}_3,f_3)$, $x$ is unoriented.}
\end{enumerate}

This analysis shows that when $v$ is not in $N(u)$ in $(\mathcal{G},f)$, then $(u,\; v,\; w,\; x)$ is not an oriented sequence.

Also in Case (B) a similar case analysis exhibits a contradiction with the assumption that $(u,\; v,\; w,\; x)$ is an oriented sequence.

Thus, existence of a non-maximal oriented sequence $(u,v,w)$ leads to a contradiction. 
\end{proof}

\begin{theorem}[The \text{\cds} Rescue Theorem]\label{cdsrescue}
Let $\pi$ be a \cdr-sortable signed permutation.  Then all \cdr fixed points of $\pi$ are \cds-sortable.
\end{theorem}
\begin{proof}
Let $\pi$ be a \cdr sortable signed permutation. Then the overlap graph of $\pi$ has no unoriented components. We prove this result by induction on the positive integer $n$   for which $\pi$ is a member of $\textsf{S}^{\pm}_n$.  For $n=1$ there is nothing to prove. 

Thus, assume $n>1$ and that the statement of the theorem has been confirmed for all $k<n$. Let  $\pi\in \textsf{S}^{\pm}_n$ be a \cdr-sortable permutation with no adjacencies, and with some (non-sorted) \cdr-fixed point $\beta$. Consider an oriented sequence $(a_1,\cdots,a_{m})$ of pointers used successively to obtain $\beta$ from $\pi$. 

{\flushleft{\underline{Case 1: $m=1$.}}} Then $(a_1)$ is a maximal oriented sequence for the oriented overlap graph $(\mathcal{G},f)$ associated with $\pi$. Now by the Fundamental Theorem of Oriented Graphs, Theorem \ref{gcdrtheorem}, choose an oriented vertex $r_1$ of $\pi$ such that the overlap graph of $\gamma_1 = \text{\cdr}_{r_1}(\pi)$ has no unoriented components, and then choose an oriented vertex $r_2$ of $\gamma_1$ so that the overlap graph of $\gamma_2 = \text{\cdr}_{r_2}(\gamma_1)$ has no unoriented components. Then $\gamma_2$ is \cdr sortable and by Lemma \ref{maxextend} $(r_1,\, r_2,\, a_1)$ is a maximal oriented sequence of pointers for $\pi$, whence $(a_1)$ is a maximal oriented sequence of pointers for $\gamma_2$. Now $\gamma_2$ has at least two adjacencies. Collapsing these adjacencies produces a signed permutation $\gamma_2^{*}$ which has a maximal oriented sequence $(a_1^{*})$ of pointers, and is a member of $\textsf{S}_{k}^{\pm}$ for a $k<n$. Applying $\text{\cdr}_{a_1^{*}}$ to $\gamma^{*}$ produces a \cdr fixed point $\delta^{*}$ of $\gamma^{*}$ which, by the induction hypothesis, is \cds sortable. Reinstating the adjacencies we find that $\delta$ is a \cdr fixed point of $\text{\cdr}_{a_1}(\gamma_2)$, and is \cds sortable. However, with $\beta = \text{\cdr}_{a_1}(\pi)$, we have by Lemma \ref{cdsrescuelemma} that $\delta = \text{\cds}_{r_1,r_2}(\beta)$, and it follows that $\beta$ is \cds sortable.

{\flushleft{\underline{Case 2: $m>1$.}}} The induction hypothesis is that the statement is true for permutations $\delta$ in $\textsf{S}^{\pm}_k$, $k <n$, and all oriented sequences. 
{\flushleft{Subcase 2 (a): }} The overlap graph of $\pi_1 = \textbf{cdr}_{a_1}(\pi)$ has no unoriented component. In this case, as $\pi_1$ has an adjacency, we may reduce $\pi_1$ through collapsing the adjacency, to an equivalent permutation $\pi^{*}_1$ in $\textsf{S}_{n-1}^{\pm}$ in which the oriented sequence corresponding to $(a_2,\;\cdots,\; a_m)$ sorts the equivalent permutation to the fixed point $\beta^{*}$ corresponding to $\beta$ under the same collapse of the same adjacency. By the induction hypothesis $\beta^*$ is \cds-sortable, and thus $\beta$ is \cds-sortable.

{\flushleft{Subcase 2 (b): }} The overlap graph of $\pi_1 = \textbf{cdr}_{a_1}(\pi)$ has an unoriented component. Restrict attention to the set of those vertices of the overlap graph of $\pi$ that appear in the unoriented components of the overlap graph of $\pi_1$, together with the vertex $a_1$. Let this set of vertices be $V_1$. The vertex subgraph of $(\mathcal{G}_1,f_1)$ the oriented graph of $\pi$ induced by the set of vertices $V_1$ is an oriented graph which has no unoriented components, and $(a_1)$ is a maximal sequence in this subgraph. By the Fundamental Theorem of Oriented Graphs, Theorem \ref{gcdrtheorem}, we find an oriented vertex $r_1\in V_1$ such that $(\mathcal{G}_2,f_2) = \text{\bf gcdr}((\mathcal{G}_1,f_1),r_1)$ has no unoriented components. Observe that $r_1 \neq a_1$, and that in the latter graph $a_1$ is an unoriented vertex. By yet another application of the Fundamental Theorem of Oriented graphs we find a second vertex $r_2\in V_1\setminus \{r_1\}$, oriented in $(\mathcal{G}_2,f_2)$, such that $\text{\bf gcdr}((\mathcal{G}_2,f_2) ,r_2)$ has no unoriented components.

By Lemma \ref{maxextend} the sequence $(r_1,\, r_2,\, a_1)$ is maximal in $({\mathcal G}_1,f_1)$. Since the sequence $(r_1,\, r_2)$ produces no unoriented components in this graph, the argument in the proof of Theorem 3 of \cite{TBS} shows that this sequence, also oriented in the overlap graph of $\pi$, produces no unoriented components from the overlap graph of $\pi$.

Consider $\gamma_1 = \text{\cdr}_{r_1}(\pi)$. Then $\gamma_1$ is an element of $\textsf{S}^{\pm}_n$ with an adjacency, and the overlap graph of $\gamma_1$ has no unoriented components. By the induction hypothesis, and the argument in Subcase 2 (a), the \cdr fixed point of $\gamma_1$ arising from applications of \cdr using the sequence of pointers $(r_2,\, a_1,\, \cdots,\, a_m)$ is \cds-sortable. By Lemma \ref{cdsrescuelemma} the result of applying the sequence of \cdr operations corresponding to $(r_1,\, r_2,\, a_1,\, \cdots,\, a_m)$ to $\pi$ results in $\text{\cds}_{r_1,r_2}(\beta)$. As the latter is \cds-sortable, so is $\beta$.

This completes the proof.
\end{proof}

\begin{repexample}{ex:maxseq}
We illustrate Theorem \ref{cdsrescue} with  $\pi = \lbrack 1,\, 3,\, 5,\, -2,\, -6,\, 4 \rbrack$. 
\end{repexample}
As noted earlier, $\pi$ is \cdr sortable and application of $\text{\cdr}_{(5,6)}$ to $\pi$ produces $\gamma = \lbrack 1,\, 3,\, 5,\, 6,\, 2,\, 4 \rbrack$ to which no further \cdr can be applied. Observe that $\text{\cds}_{(1,2),(2,3)}$ is applicable to $\gamma$, and that $\delta =  \lbrack 1,\, 2,\, 3,\, 5,\,6,\, 4\rbrack = \text{\cds}_{(1,2),(2,3)}(\gamma)$. Next, $\text{\cds}_{(3,4),(4,5)}$ is applicable to $\delta$ and produces the identity permutation. The \cdr fixed point $\gamma$ of the signed permutation $\pi$ is \cds sortable.\\

As a second application of these methods and results we find:
\begin{theorem}[\cdr Steps Theorem]\label{indiscriminatecdr} Let $\pi$ be a \cdr sortable signed permutation. Suppose that $k$ applications of \cdr produces a \cds fixed point $\beta$ of $\pi$, and that $m$ applications of \cds to $\beta$ results in the identity permutation. Then $\pi$ is \cdr sortable by $k+2m$ applications of \cdr.
\end{theorem}
Thus, the number of applications of \cdr required to sort a \cdr sortable signed permutation can be computed by indiscriminate applications of \cdr and \cds.
Theorem \ref{indiscriminatecdr} also finally explains the 2-to-1 ratio in \cdr and \cds operations observed in \cite{HNS} in constructing phylogenies among different Muller elements for several species of fruitflies, and justifies using this ratio in constructing distance matrices towards the phylogenetic analysis - see pp. 14-15, footnote 11 and for example Figures 9 and 10 of \cite{HNS}.  

\section{The \cdr parity theorem and combinatorial games}

It has been proven in \cite{AHMMSTW} that for each permutation, if two sequences of applications of \cds lead to a \cds fixed point, then these two sequences have the  same length. This fails for the sorting operation \cdr. However, for \cdr the parity of the lengths of sequences of applications of \cdr leading to \cdr fixed points of a given permutation $\pi$ is an invariant. 

\begin{theorem}[The \textbf{cdr} Parity Theorem]\label{cdrparity}
For each signed permutation $\pi$, the lengths of maximal sequences of pointers are of the same parity. 
\end{theorem}
\begin{proof}
Let a signed permutation $\alpha\in\textsf{S}^{\pm}_n$ be given.

{\flushleft\underline{Case 1: $\alpha$ is  \textbf{cdr} sortable: }} 
First note that any two total sequences of oriented vertices of the corresponding overlap graph are of the same length, by Theorem \ref{cdrfundthm}. Consider two sequences of pointers, say $S_1, S_2$, that each is maximal. Applications of \cdr for these pointers thus result in \cdr-fixed points $\beta_1$ and $\beta_2$ respectively. By Corollary \ref{tbscorollary} we can extend each of $S_1$ and $S_2$ individually by two pointers at a time until a total sequence of pointers is reached, meaning the identity permutation results from applications of \cdr. By adding an even number of terms to a sequence, we do not change the parity of the lengths of these sequences. Thus, as the two sequences ${S_1}', {S_2}'$ which sort the permutation are the same length, and thus these lengths have same parity, also the lengths of $S_1$ and $S_2$ are the same parity. 

{\flushleft\underline{Case 2: $\alpha$ is not {\textbf{cdr}}-sortable.}} 
By Theorem \ref{cdrsortfundthm} the overlap graph of $\alpha$ has unoriented components. The only pointers to which \cdr can be applied are vertices in the oriented components of the overlap graph of $\alpha$. Consider some $\alpha`$ that has the same move graph as $\alpha$ without the unoriented component. Then $\alpha'$ has the same sequences of moves possible as $\alpha$. As $\alpha'$ does not have an unoriented component, Case 1 applies, whence lengths of all sequences for $\alpha'$ are of the same parity. 
\end{proof}

Example \ref{ex:maxseq} provides an illustration of the \cdr parity phenomenon: \cdr fixed points are reached after 1, 3 or 5 applications of \cdr. The \cdr Parity Theorem can also be applied to a question about combinatorial games based on \cdr, defined for signed permutations in \cite{AHMMSTW}. 

In the oriented graph context these games are defined as follows:  Let an oriented graph $(\mathcal{G},f)$ on a finite set of vertices be given.  Player ONE starts the game by selecting an oriented vertex $o_1$ and computing the graph $(\mathcal{G}_1,f_1) = \textbf{gcdr}((\mathcal{G},f), o_1)$. Then player TWO selects an oriented vertex $t_1$ of $(\mathcal{G}_1,f_1)$ and computes the graph $(\mathcal{G}_2,f_2) = \textbf{gcdr}((\mathcal{G}_1,f_1),t_1)$, and so on. The game continues until no  more legal moves are possible - \emph{i.e.}, there are no more oriented vertices left in the graph. 

The \emph{normal play} rule version of the game is denoted $\textsf{N}((\mathcal{G},f),\text{\bf gcdr})$. In this game, the player able to make the last legal move wins. In the special case when the graph $(\mathcal{G},f)$ is the oriented overlap graph of a signed permutation $\alpha$, this game is denoted $\textsf{N}(\alpha,\text{\cdr})$. In the \emph{misere play} version, the player making the last legal move looses. The corresponding notation for the misere play games is $\textsf{M}((\mathcal{G},f),\textbf{gcdr})$ or $\textsf{M}(\alpha,\textbf{cdr})$.

Thus, the normal and misere version of the \cdr game on signed permutations are solved: In order to determine who has the winning strategy, just determine the length of a single play of the game. This shows that decision problems \textbf{D.6} and \textbf{D.7} of \cite{AHMMSTW} are linear time decision problems.

\section{The ciliates}

According to the model described in \cite{PER1} the micronuclear precursors of genes in ciliates should be sortable or reverse sortable by \cdr or \cds to produce functional genes. For some indication of the scope of sorting required during ciliate micronuclear decryption, consider the following: Findings reported in \cite{CL} indicate that at least 3593 genes on 2818 chromosomes of the ciliate \emph{Sterkiella histriomuscorum} have encrypted micronuclear precursors. 1676 of these encrypted precursors contain at least one inverted element. \cite{CL} also reports an example of an encrypted precursor organized into 245 precursor segments.  

Sample genes from ciliates show that both the \cdr and the \cds operations are necessary to sort or reverse sort signed permutations representing micronuclear precursors of genes.
We discuss a small sample of micronuclear precursors, and their sortability by \cdr and \cds. More examples can be found at the ciliate genome rearrangement database \cite{mdsiesdb}.  

\begin{center}{\bf Example: The Actin I gene}
\end{center}

{\flushleft{\underline{\emph{Oxytricha nova}:} }}
Decryption of the micronuclear precursor of the Actin I gene in the ciliate species \emph{Oxytricha nova}, reported in \cite{GPOLC} (see Fig. 3 there),  requires sorting of the signed permutation $\alpha = \lbrack 3,\; 5,\; 4,\; 6,\; 8,\; -2,\; 1,\; 7\rbrack$. The overlap graph of $\alpha$, shown in Figure \ref{fig:ONovaActinIOverlapGraph}, has an unoriented component. $\alpha$ is neither \cdr sortable, nor \cds-sortable, but requires both \cdr and \cds for sorting.
\begin{figure}[h]
\centering
\begin{tikzpicture}[scale = .15]
\coordinate (red) at (0,-10);
\coordinate (orange) at (-8,-6);
\coordinate (yellow) at (-10,2);
\coordinate (green) at (-4,9);
\coordinate (blue) at (4, 9);
\coordinate (purple) at (10,2);
\coordinate (gray) at (8,-6);

\node at (green) [above] {(1,2)};
\node at (blue) [above] {(2,3)};
\node at (purple) [right] {(3,4)};
\node at (gray) [right] {(4,5)};
\node at (red) [below] {(5,6)};
\node at (orange) [left] {(6,7)};
\node at (yellow) [left] {(7,8)};

\draw [thick, fill = black] (green) circle (15 pt);
\draw [thick, fill = black] (blue) circle (15 pt);
\draw [thick, fill = white] (purple) circle (15 pt);
\draw [thick, fill = white] (gray) circle (15 pt);
\draw [thick, fill = white] (red) circle (15 pt);
\draw [thick, fill = white] (orange) circle (15 pt);
\draw [thick, fill = white] (yellow) circle (15 pt);
\draw [thick] (blue) -- (yellow);
\draw [thick] (yellow) -- (orange);
\draw [thick] (orange) -- (blue);
\draw [thick] (gray) -- (purple);
\draw [thick] (red) -- (gray);
\draw [thick] (purple) -- (red);

\end{tikzpicture}
\caption{The overlap graph of the Actin I micronuclear precursor in \emph{O. nova} }\label{fig:ONovaActinIOverlapGraph}
\end{figure}
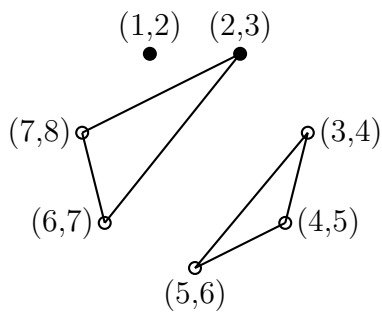
Applications of \cdr terminate in the permutation $\lbrack -8,\; -7,\; -6,\; -4,\; -5,\; -3,\; -2,\; -1\rbrack$, which has item $-5$ out of order. An application of \cds produces the permutation $\lbrack -8,\; -7,\; -6,\; -5,\; -4,\; -3,\; -2,\; -1\rbrack$, which represents a functional Actin I gene in the organism.

{\flushleft{\underline{\emph{Uroleptus pisces}:} }}
The structures of the micronuclear precursors of the Actin I gene for \emph{U. pisces} 1 and \emph{U. pisces} 2, reported in \cite{DP}, were given in the introduction. We treat $\alpha_1 = \lbrack 1,\; 3,\;  -7,\;  -5,\;  14,\;  2,\;  4,\;  6,\;  9,\;  12,\;  -11,\;  -8,\;  13,\;  15,\;  -10\rbrack$, the precursor for \emph{U. pisces} 1. 
To determine \cdr sortability of $\alpha_1$ we construct the overlap graph of $\alpha_1$, shown in Figure \ref{fig:OverlapGraphUp12}. 
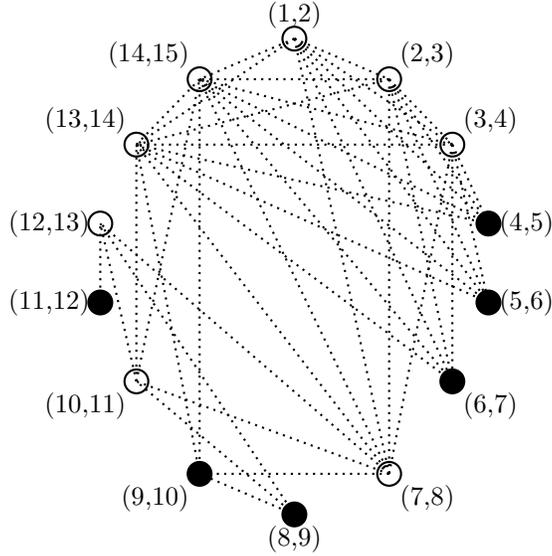
\begin{figure}[ht]
\begin{tikzpicture}[scale = .30]

\coordinate (p12) at (0.0,7.7);
\coordinate (p23) at (4.2,5.9); 
\coordinate (p34) at (7.0,3.0); 
\coordinate (p45) at (8.6,-0.5); 
\coordinate (p56) at (8.6,-4.0); 
\coordinate (p67) at (7.0,-7.5); 
\coordinate (p78) at (4.2,-11.6);
\coordinate (p89) at (0.0,-13.4); 
\coordinate (p910) at (-4.2,-11.6);
\coordinate (p1011) at (-7.0,-7.5); 
\coordinate (p1112) at (-8.6,-4.0); 
\coordinate (p1213) at (-8.6,-0.5);
\coordinate (p1314) at (-7.0,3.0); 
\coordinate (p1415) at (-4.2,5.9);

\node at (p12) [above] {\footnotesize (1,2)};
\node at (p23)  [anchor=north east,above right]{\footnotesize (2,3)};
\node at (p34) [anchor=north east,above right] {\footnotesize (3,4)};
\node at (p45) [right] {\footnotesize (4,5)};
\node at (p56) [anchor=south east, right] {\footnotesize (5,6)};
\node at (p67) [anchor=south east, below right] {\footnotesize (6,7)};
\node at (p78) [below right] {\footnotesize (7,8)};
\node at (p89) [below] {\footnotesize (8,9)};
\node at (p910) [below left] {\footnotesize (9,10)};
\node at (p1011) [below left] {\footnotesize (10,11)};
\node at (p1112) [left] {\footnotesize (11,12)};
\node at (p1213) [left] {\footnotesize (12,13)};
\node at (p1314) [above left] {\footnotesize (13,14)};
\node at (p1415) [above left] {\footnotesize (14,15)};

\draw [thick] (p12) circle (15 pt);
\draw [thick] (p23) circle (15 pt);
\draw [thick] (p34) circle (15 pt);
\draw [thick, fill = black] (p45) circle (15 pt);
\draw [thick, fill = black] (p56) circle (15 pt);
\draw [thick, fill = black] (p67) circle (15 pt);
\draw [thick] (p78) circle (15 pt);
\draw [thick, fill = black] (p89) circle (15 pt);
\draw [thick, fill = black] (p910) circle (15 pt);
\draw [thick] (p1011) circle (15 pt);
\draw [thick, fill = black] (p1112) circle (15 pt);
\draw [thick] (p1213) circle (15 pt);
\draw [thick] (p1314) circle (15 pt);
\draw [thick] (p1415) circle (15 pt);

\draw [thick, black, dotted] (p12) -- (p23);%
\draw [thick, black, dotted] (p12) -- (p34);%
\draw [thick, black, dotted] (p12) -- (p45);%
\draw [thick, black, dotted] (p12) -- (p56);%
\draw [thick, black, dotted] (p12) -- (p67);%
\draw [thick, black, dotted] (p12) -- (p78);%
\draw [thick, black, dotted] (p12) -- (p1314);%
\draw [thick, black, dotted] (p12) -- (p1415);%
\draw [thick, black, dotted] (p23) -- (p34);%
\draw [thick, black, dotted] (p23) -- (p45);%
\draw [thick, black, dotted] (p23) -- (p56);%
\draw [thick, black, dotted] (p23) -- (p67);%
\draw [thick, black, dotted] (p23) -- (p78);%
\draw [thick, black, dotted] (p23) -- (p1314);%
\draw [thick, black, dotted] (p23) -- (p1415);%
\draw [thick, black, dotted] (p34) -- (p45);%
\draw [thick, black, dotted] (p34) -- (p56);%
\draw [thick, black, dotted] (p34) -- (p67);%
\draw [thick, black, dotted] (p34) -- (p78);%
\draw [thick, black, dotted] (p34) -- (p1314);%
\draw [thick, black, dotted] (p34) -- (p1415);%
\draw [thick, black, dotted] (p45) -- (p1314);%
\draw [thick, black, dotted] (p45) -- (p1415);%
\draw [thick, black, dotted] (p56) -- (p1314);%
\draw [thick, black, dotted] (p56) -- (p1415);%
\draw [thick, black, dotted] (p67) -- (p1314);%
\draw [thick, black, dotted] (p67) -- (p1415);%
\draw [thick, black, dotted] (p78) -- (p910);%
\draw [thick, black, dotted] (p78) -- (p1011);%
\draw [thick, black, dotted] (p78) -- (p1213);%
\draw [thick, black, dotted] (p78) -- (p1314);%
\draw [thick, black, dotted] (p78) -- (p1415);%
\draw [thick, black, dotted] (p89) -- (p910);%
\draw [thick, black, dotted] (p89) -- (p1011);%
\draw [thick, black, dotted] (p89) -- (p1213);%
\draw [thick, black, dotted] (p910) -- (p1314);%
\draw [thick, black, dotted] (p910) -- (p1415);%
\draw [thick, black, dotted] (p1011) -- (p1213);%
\draw [thick, black, dotted] (p1011) -- (p1314);%
\draw [thick, black, dotted] (p1011) -- (p1415);%
\draw [thick, black, dotted] (p1112) -- (p1213);%
\draw [thick, black, dotted] (p1314) -- (p1415);%

\end{tikzpicture}
\caption{Overlap graph for $\alpha_1$. Filled vertices are  oriented.}\label{fig:OverlapGraphUp12}
\end{figure}

The overlap graph of $\alpha_1$ 
is an oriented graph and its only component is oriented. By Theorem \ref{cdrsortfundthm} $\alpha_1$ is \cdr sortable.  As the reader may verify, applying \cdr in order for the following ordered length 14 sequence of pointers accomplishes such a sorting:
\[{\tiny{
  (4, 5)\, (3, 4)\, (2, 3)\, (5, 6)\, (6, 7)\, (1, 2)\, (7, 8)\, (8, 9)\, (9, 10)\, (11, 12)\, (12, 13)\, (13, 14)\, (14, 15)\, (10, 11)}}
\]
By Theorem \ref{cdrfundthm} any sequence of applications of \cdr that sorts permutation $\alpha_1$ will use exactly fourteen pointers. But innocent-looking deviations from a sequence of pointers supporting a successful \cdr sorting of $\alpha_1$ may result in a failed sorting. For example, if in the sequence of pointers above right after \cdr has been applied to pointer (12,13) it is applied to pointer (10,11) (supporting a legitimate \cdr application at that stage), the result will be the unsorted permutation
\[
  \beta = \lbrack
  1\, 2\, 3\, 4\, 5\, 6\, 7\, 8\, 9\, 10\, 11\, 12\, 13\, 15\, 14
   \rbrack
\]
How does the ciliate decryptome resolve this problem?  The \cds~{\tt Rescue Theorem}, Theorem \ref{cdsrescue}, proves that the \cdr fixed points encountered are in fact \cds sortable. In this particular example, $\text{\cds}_{(13,14),(14,15)}(\beta)$ is the identity permutation. Thus the \cds sorting operation assures that decryption of $\alpha_1$ does not fail. Moreover, this theorem and the \cds Inevitability Theorem imply that \emph{indiscriminate} applications of \cdr and of \cds will succeed in sorting $\alpha_1$ (and any \cdr sortable micronuclear gene pattern). 

\begin{center}{\bf Example: $\alpha$TBP gene in \emph{O. trifallax}}
\end{center}

For some permutations that have been reported in extant ciliate species, \cds is both necessary and sufficient: For example the micronuclear precursor of the $\alpha$ Telomere Binding Protein gene in \emph{Oxytricha trifallax} is represented by the permutation $\lbrack 1,\; 3,\; 5,\; 7,\; 9,\; 11,\; 2,\; 4,\; 6,\; 8,\; 10,\; 12\rbrack$ has this feature. This permutation was reported in \cite{PDP}.

\begin{center}{\bf DNA polymerase $\alpha$ gene in \emph{O. trifallax}}
\end{center}

Likewise, for some signed permutations \cdr is both necessary and sufficient to sort or reverse sort the corresponding permutation, and no application of \cds would contribute to its sorting or reverse sorting. For this discussion, define for each $n$ the signed permutations
\[
  \sigma_n = \lbrack -2n,\; \cdots,\; -2,\; 1,\; 3,\; \cdots,\; 2n-1,\; 2n+1 \rbrack.
\]
and 
\[
  \tau_n = \lbrack -2n,\; \cdots,\; -2,\; 1,\; 3,\; \cdots,\; 2n-1 \rbrack.
\]

The oriented overlap graph of $\sigma_n$ consists of $2n$ isolated oriented vertices, while the oriented overlap graph of $\tau_n$ consists of $2n-1$ isolated oriented vertices. Observe that any pointer is eligible for a \cdr move, and that this oriented graph has isolated vertices only. Thus, no applications of \cds to any intermediaries contributes to the sorting process.  The signed permutations $\sigma_n$ and $\tau_n$ also have the feature that for any selection of a pointer, a corresponding application of \cdr can be performed. Each $\sigma_n$ is \cdr sortable, while each $\tau_n$ is \cdr reverse-sortable. 

These permutations, intriguingly, occur in the micronuclear precursors of macro nuclear genes of certain ciliate species. For some species of \emph{Uroleptus}, the DNA polymerase $\alpha$ gene is $\sigma_{16}$, while for \emph{Paraurystola weissei} the same precursor is $\sigma_{20}$. In \emph{Oxytricha nova} the micronuclear precursor is
$\tau_{21}$.
 In \emph{Oxytricha trifallax} the micronuclear precursor is $\sigma_{21}$. 
Accession numbers for these DNA polymerase $\alpha$ macronuclear and corresponding micronuclear precursors in \cite{genbank} are\\
\begin{center}
\begin{tabular}{|l|l|l|l|} \hline
Species & Macronuclear & Micronuclear \\ \hline
Uroleptus &  AY293852 & AY293850  \\ \hline
Paraurystola weissei & Y293806 & AY293805  \\ \hline
Oxytricha trifallax & EF693893.1  & DQ525914.1   \\ \hline
\end{tabular}
\end{center}
\vspace{0.1in}

\section{Concluding Remarks}

Unless the applications of \textbf{cdr} operations during micronuclear decryption follow a strategy not yet discovered in the laboratory, it should be expected that \textbf{cdr} fixed points other than the identity are often encountered among intermediates of the decryption process. It would be interesting to determine whether \cdr fixed points of \cdr sortable micronuclear gene patterns actually do occur among intermediates of the decryption process.

Results reported in \cite{CL} (see Figure 3 A, B and C there) suggest that the model of \cite{PER, PER1}  based on just the \cdr and \cds sorting operations might require an additional operation to successfully sort newly observed encryption patterns.  The permutation representing the red micronuclear precursor in Figure 3 B,  and the permutation representing the gold micronuclear precursor in Figure 3 C are \cds fixed points. It would be interesting to learn how these two particular micronuclear precursors are in fact processed by the ciliate decryptome. In \cite{EHPPR} and some earlier papers it is assumed that besides \cdr and \cds, there is an additional sorting operation, \emph{boundary} \textsf{ld}, that would sort a \cds fixed point to the identity permutation. To our knowledge it has not been experimentally confirmed that this operation occurs during ciliate micronuclear decryption, nor has a satisfactory molecular mechanism for this operation in combination with \cdr and \cds been described.

\section{Acknowledgements}

The research represented in this paper was funded by an NSF REU grant DMS 1359425, by Boise State University and by the Department of Mathematics at Boise State University.

\end{document}